\documentclass[reqno]{amsart}
\usepackage{amssymb, bm, amsmath, latexsym, mathabx, dsfont,tikz, float, mathrsfs}
\usepackage{makecell, multirow, longtable, kotex}
\usepackage{enumitem}
\usepackage{tikz}

\renewcommand{\baselinestretch}{\baselinestretch}
\renewcommand{\baselinestretch}{1.1}
\numberwithin{equation}{section}

\newtheorem{thm}{Theorem}[section]
\newtheorem{lem}[thm]{Lemma}
\newtheorem{cor}[thm]{Corollary}
\newtheorem{prop}[thm]{Proposition}
\newtheorem{conj}[thm]{Conjecture}

\theoremstyle{definition}

\theoremstyle{remark}
\newtheorem{rmk}[thm]{Remark}

\numberwithin{equation}{section}

\newcommand{\cls}{\text{cls}}
\newcommand{\pcls}{\text{cls}^+}

\newcommand{\ra}{{\, \rightarrow \,}}

\newcommand{\z}{{\mathbb Z}}

\newcommand{\q}{{\mathbb Q}}

\newcommand{\n}{{\mathbb N}}

\newcommand{\oo}{{\mathcal{O}}}

\newcommand{\of}[1]{{\mathcal{O}_F^{#1}}}
\newcommand{\Mod}[1]{\ (\mathrm{mod}\ #1)}

\newcommand{\mfa}{{\mathfrak{a}}}

\newcommand{\Q}{\mathbb{Q}}
\newcommand{\R}{\mathbb{R}}
\newcommand{\C}{\mathbb{C}}
\newcommand{\Z}{\mathbb{Z}}
\newcommand{\N}{\mathbb{N}}
\newcommand{\GL}{\mathrm{GL}}
\newcommand{\SL}{\mathrm{SL}}

\newcommand{\slzeta}[2]{\zeta_{#1}^{\SL}\left(#2\right)}
\newcommand{\glzeta}[2]{\zeta_{#1}^{\GL}\left(#2\right)}
\newcommand{\Cl}{\mathrm{Cl}}

\usepackage[colorinlistoftodos]{todonotes}

\begin{document}

	
	\author{Daejun Kim}
	\address{Department of Mathematics Education, Korea University,
		Seoul 02841, Republic of Korea}
	\email{daejunkim@korea.ac.kr}	
	\thanks{}
	
	\author{Seok Hyeong Lee}
	\address{Center for Quantum Structures in Modules and Spaces, Seoul National University, Seoul 08826, Republic of Korea}
	\thanks{The second author was supported by the National Research Foundation of Korea(NRF) grant funded by the Korea government(MSIT) (No.2020R1A5A1016126).}
	\email{lshyeong@snu.ac.kr}

 \author{Seungjai Lee}
	\address{Department of Mathematics, Incheon National University, Incheon 22012, Republic of Korea}
	\email{seungjai.lee@inu.ac.kr}

	
	
	\thanks{}
	
	
	\title[]{Zeta functions enumerating subforms of quadratic forms}
	
	\begin{abstract} 
		In this paper, we introduce and study the Dirichlet series enumerating (proper) equivalence classes of full rank subforms/sublattices of a given quadratic form/lattice, focusing on the positive definite binary case. We obtain formulas linking this Dirichlet series with Dirichlet series counting ideal classes of the imaginary quadratic field associated with the quadratic form. Utilizing the result, we provide explicit formulas of the Dirichlet series for several lattices, including square lattice and hexagonal lattice.  Moreover, we investigate some analytic properties of this Dirichlet series.

	\end{abstract}
	\maketitle
	
	\section{Introduction}

In the study of quadratic forms, there are several zeta functions associated with a quadratic form or a vector space of quadratic forms. Two notable examples are the zeta functions of indefinite quadratic forms developed by Siegel \cite{Sie1} and the zeta functions associated with the vector space of quadratic forms developed by Shintani \cite{Shin1}. We refer the readers a nice short survey note written by Ibukiyama \cite{Ibuki1} on various zeta functions of quadratic forms, which includes brief overviews of \cite{Ibuki-Katsu} and \cite{Ibuki-Saito}.

The aim of this article is to  introduce a new zeta function associated with a quadratic form. Specifically,  we study the Dirichlet series enumerating (proper) equivalence classes of full rank subforms of a given quadratic form $f$, where we will denote them as $\slzeta{f}{s}$ and $\glzeta{f}{s}$.

Quadratic forms and their associated structures, such as quadratic spaces and lattices, hold significant importance in mathematics and its applications.  One particular topic of interest is crystallography, where lattices in $\R^2$ and $\R^3$ have been studied for more than a century to understand symmetries of periodic crystals. As a consequence, classifying and enumerating sublattices with certain symmetric constraints has been extensively studied; see \cite{BSZ1, BZ1, BSW1, Rut1, Rut2} and references therein.

Since the Dirichlet series we defined in this article also enumerate (proper) isometry classes of sublattices of finite index of a given lattice, our results also provide new insights and extend existing theories on the Dirichlet series enumerating crystallographic objects. For example, our explicit examples in Section \ref{subsec:hexagon} and \ref{subsec:square} correspond to counting inequivalent sublattices of fixed index under rotations and reflections in hexagonal lattice and square lattice, respectively.

A notable difference between our results and existing crystallographic formulas, such as those in \cite[Table 2]{Rut2}, is our broader equivalence criteria. Crystallographic studies like \cite{Rut2} consider rotations and reflections that preserve the parent lattice, involving a finite number of symmetries. In contrast, our approach counts inequivalent sublattices under any rotations or reflections, encompassing an infinite number of symmetries. This approach is both more natural and more challenging from a mathematical perspective (cf. Remark \ref{rmk:challenge}). For instance, the sequence $(a_{m})$ from our formula for $\glzeta{x^2+xy+y^2}{s}$ in Section \ref{subsec:hexagon} matches the sequence A300651 in the OEIS (On-line Encyclopedia of Integer Sequences), whereas the formula for ``$p6mm$'' in \cite[Table. 2]{Rut2} and the formula in \cite{BSW1} correspond to A003051 in OEIS. Furthermore, our methodology can extend to arbitrary lattices in $\R^d$, not limited to those with real-world symmetries of periodic crystals.

Quadratic forms and lattices are fundamental objects in mathematics with increasingly many applications. We aim for this article to serve as a comprehensive introduction to the theory of zeta functions of (proper) isometry classes of a lattice, encouraging further exploration and advancement in this field.

\subsection{Main Results}

To describe our object precisely, we introduce the geometric language of quadratic spaces and lattices. Let us consider an $n\times n$ symmetric matrix $S=(s_{ij})\in \text{Sym}_n(\q)$ with $s_{ij}\in\frac{1}{2}\z$ and $s_{ii}\in \z$. This gives rise to a quadratic form with integer coefficients $Q(\bm{x})=\bm{x}^t S \bm{x}$ on $\z^n$. We may equivalently consider a $\z$-lattice $L=\z e_1+\cdots+\z e_n$ of rank $n$ on the quadratic space $V$ over $\q$ equipped with quadratic form $Q:V\ra\q$ and associated non-degenerate symmetric bilinear form $B:V\times V \ra \q$ with $2B(x,y)=Q(x+y)-Q(x)-Q(y)$ and obtain $S$ as the {\em Gram matrix} $M_L=(B(e_i,e_j))$ of $L$ with respect to the basis $\{e_1,\ldots,e_n\}$. The determinant of a Gram matrix $M_L$ of $L$ is called the {\em discriminant} of $L$, and we denote it by $d_L$. The quadratic form $\bm{x}^t M_L \bm{x}$ corresponding to $L$ is denoted by $Q_L(\bm{x})$. 

The orthogonal group $O(V)$ and the special orthogonal group $O^+(V)$ of $V$ are defined by
\[
O(V)=\{\sigma\in \GL(V) : B(\sigma x,\sigma y)=B(x,y) \text{ for any }x,y\in V\} 
\quad \text{and} \quad 
O^+(V)=O(V)\cap \SL(V).
\]
The orthogonal group $O(V)$ acts on the set of $\z$-lattices on $V$. We say two $\z$-lattices $L_1$ and $L_2$ on $V$ are {\em isometric} if $\sigma(L_1)=L_2$ for some $\sigma\in O(V)$, and {\em properly isometric} if $\sigma(L_1)=L_2$ for some $\sigma\in O^+(V)$. Note that if we let $S_1$ and $S_2$ be Gram matrices of $L_1$ and $L_2$, respectively, then we have $\sigma(L_1)=L_2$ for some $\sigma\in O(V)$ (resp. $\sigma\in O^+(V)$) if and only if $S_1=T^t S_2 T$ for some $T\in \GL_n(\z)$ (resp. $T\in \SL_n(\z)$).

  Let $L$ be a $\z$-lattice on a quadratic space $V$. For $m\in\n$, let $\mathcal{L}_m$ denote the set of all $\z$-sublattices of $L$ of index $m$: 
    \[
        \mathcal{L}_m:=\{ K \subseteq L \mid \lvert L:K\rvert=m\},
    \]
    where $\lvert L:K \rvert$ denotes the cardinality of the $\Z$-module $L/K$.
Let $a_m^+(L)$ denote the number of $O^+(V)$-inequivalent lattices in $\mathcal{L}_m$, and $a_m(L)$ denote the number of $O(V)$-inequivalent lattices in $\mathcal{L}_m$.

We define the \emph{zeta function of proper isometry classes of $L$}, denoted by $\slzeta{L}{s}$, and the \emph{zeta function of isometry classes of $L$}, denoted by $ \glzeta{L}{s}$,  to be the Dirichlet series 
\[
    \slzeta{L}{s}:=\sum_{m=1}^\infty a_m^+(L) m^{-s} \qquad \text{and} \qquad   \glzeta{L}{s}:=\sum_{m=1}^\infty a_m(L) m^{-s},
\]
enumerating $O^+(V)$-equivalence and $O(V)$-equivalence classes of sublattices of finite index in $L$, respectively. (Here, and throughout this article,  $s$ is a complex variable.) Equivalently, for a quadratic form with integer coefficients $f(\bm{x})=Q_L(\bm{x}):=\bm{x}^t M_L \bm{x}$ corresponding to $L$,  we define the \emph{zeta function of proper equivalence classes of $f$}, denoted by $\slzeta{f}{s}$, and \emph{zeta functions of equivalence classes of $f$}, denoted by  $\glzeta{f}{s}$, as $\slzeta{f}{s}=\slzeta{L}{s}$ and $\glzeta{f}{s}=\glzeta{L}{s}$, which become the Dirichlet series enumerating $\SL_n(\z)$-equivalence and $\GL_n(\z)$-equivalence classes of full rank subforms of $f$, respectively.

The aim of this article is to initiate the study of $\slzeta{L}{s}$ and $\glzeta{L}{s}$. Throughout this paper, we assume that any quadratic space $V$ is {\em positive definite}, that is, $Q(v)>0$ for all $v\in V\setminus\{0\}$, unless stated otherwise. We shall give the formula of $\slzeta{L}{s}$ and $\glzeta{L}{s}$ for binary $\z$-lattices $L$ whose discriminant $d_L$ satisfies that $-4d_L$ is the discriminant of an imaginary quadratic field $F$. Before presenting our formulas, we give definitions of the submodule zeta function $\zeta_{\z^2}(s)$ and the Dedekind zeta function $\zeta_F(s)$ of $F$:
\[
    \zeta_{\z^2}(s)=\sum_{K\subseteq \z^2} \lvert \z^2 :K\rvert^{-s}=\zeta(s)\zeta(s-1) \quad \text{and} \quad \zeta_{F}(s)=\sum_{\mathfrak{a}\le \of{}} \lvert\of{}:\mathfrak{a}\rvert^{-s},
\]
where the sums run over all $\z$-submodules $K\subset\z^2$ of finite index and over $\of{}$-ideals $\mathfrak{a}$ of the ring of integers $\of{}$ of $F$, respectively.
First we present the following theorems for $\slzeta{L}{s}$.

\begin{thm}\label{thm:sl.final} Let $L$ be a binary $\z$-lattice such that $-4d_L$ is the  discriminant of the imaginary quadratic field $F=\q\left(\sqrt{-4d_L}\right)$. Then we have
\[
\slzeta{L}{s} = \frac{2}{|\of{\times}|} \frac{\zeta_{\Z^2}(s)}{\zeta_{F}(s)} \left( \sum_{n=1}^{\infty} \frac{ \# \{ [J]\in \Cl_F: J \le \of{}, \, N(J)=n\} }{n^{s}} \right) +  \frac{|\of{\times}| - 2}{|\of{\times}|}  \left(\prod_{p\,\text{split}} (1-p^{-s}) \right) \zeta_{F}(s).
\]
Here $\of{}$, $\of{\times}$, $\Cl_F$ are the ring of integers, the group of units, the ideal class group of $F$, respectively, $N(J)=\lvert \of{} : J\rvert$, and the product runs over all rational primes $p$ that are split in $F$.
\end{thm}

From Theorem \ref{thm:sl.final}, we may observe that $\slzeta{L}{s}$ is uniquely determined by its corresponding imaginary quadratic field $F$. Namely, we have the following result (see Section \ref{sec:analytic} for the proof).

\begin{cor} If $L_1$ and $L_2$ are binary $\z$-lattices with $-4d_{L_1}=-4d_{L_2}$ being the  discriminant of an imaginary quadratic field, then $\slzeta{L_1}{s}=\slzeta{L_2}{s}$.
\end{cor}

There are two main ideas to prove Theorem \ref{thm:sl.final}. The first idea is to divide $\slzeta{L}{s}$ into the sum of terms $Z_{L,\rho}^+(s)$ over $\rho\in O^+(V)$, and then figure out each term. We call each term $Z_{L,\rho}^+(s)$ with $\rho\in O^+(V)$ `rotational term' (see \eqref{eqn:def-Z+} for the definition). The second idea is to view sublattices of $L$ as $\of{}$-ideals of the imaginary quadratic field $F=\q\left(\sqrt{-4d_L}\right)$ via the connection between `form class group' and `ideal class group'. We investigate how the notions of isometric lattices and proper isometries are converted when we deal with $\of{}$-ideals, and use them to calculate each rotational term. These connections will be described in Section \ref{sec:binaryform-ideal-corresp} in detail.

Similar ideas are used in obtaining the formulas for $\glzeta{L}{s}$ with the replacement of $Z_{L,\rho}^+(s)$ by $Z_{L,\rho}(s)$ (see \eqref{eqn:def-Z} for the definition). We should further investigate each term with a reflection $\rho\in O(V)\setminus O^+(V)$. The term $Z_{L,\rho}(s)$ with $\rho\in O^+(V)$ is also called `rotational term' and $Z_{L,\rho}(s)$ with $\rho\in O(V)\setminus O^+(V)$ is called  `reflection term'. Necessary results describing improper isometries between $\of{}$-ideals are studied in Section \ref{sec:glcase} and \ref{sec:gl.refl}. Utilizing them, we obtain the following formula for $\glzeta{L}{s}$:

\begin{thm}\label{thm:gl.final} Let $L$ be a binary $\z$-lattice such that $-4d_L$ is the  discriminant of the imaginary quadratic field $F=\q\left(\sqrt{-4d_L}\right)$. Then we have
\[
\glzeta{L}{s} = \mathrm{Rot}_L(s)+\mathrm{Refl}_L(s),   
\]
where the sum $\mathrm{Refl}_L(s)$ of reflection terms is given in Theorem \ref{thm:refTermGeneral} and the sum $\mathrm{Rot}_L(s)$ of rotational terms 
is given by
\[
    \mathrm{Rot}_L(s)=
\frac{1}{2} \slzeta{L}{s} + \frac{\zeta_{\Z^2}(s)}{2\zeta_{F}(s)} \sum_{n=1}^{\infty} \frac{\# \left( \{ [J]\in \Cl_F : |I : J|=n \} \setminus \{ [\overline{J}]\in \Cl_F : |I : J|=n \} \right)}{n^s}.
\]
Here, $I$ is the ideal of $\of{}$ corresponding to $L$ described in Proposition \ref{prop:binary-correspond}, $J$ is an $\of{}$-ideal contained in $I$, and the overline $\overline{J}$ above $J$ indicates the complex conjugation.
\end{thm}

As a sample, when $-4d_L\equiv 1 \Mod{4}$ with $-4d_L\neq -3$, the sum of reflection terms is given by
\[
\text{Refl}_L(s) = \frac{(1-2^{-s} + 2 \cdot 2^{-2s}) \zeta(s)^2}{2\zeta(2s) \prod_{r} (1+r^{-s})} \left( \sum_{n=1}^{\infty} \frac{ \# \{[J]\in\Cl_F : |I:J|=n, [J] = [\overline{J}] \} } {n^s} \right),
\]
where the product $\prod_r$ runs over all primes ramified in $F$ (see Section \ref{subsec:prime.splitting}).

\begin{rmk}\label{rmk:challenge}
Our main challenge in dealing with $\slzeta{L}{s}$ and $\glzeta{L}{s}$ comes from peculiarities of $O^{+}(V)$ and $O(V)$. Unlike the situation in counting inequivalent sublattices under crystallographic group actions,  both  $O^{+}(V)$ and $O(V)$ are infinite and even do not act on the set of all sublattices. In our case, $O^{+}(V)$ and $O(V)$ can be described in terms of the arithmetic of the corresponding imaginary quadratic number field, especially its prime splitting behavior and ideal class group. This explains why formulas in Theorem \ref{thm:sl.final} and Theorem \ref{thm:gl.final} are heavily dependent on these components.
\end{rmk}

The Dirichlet series $\slzeta{L}{s}$ and $\glzeta{L}{s}$ also tell us how \emph{sublattices/subforms grow}. Since $a_m^+(L)$ and  $a_m(L)$ grow at most polynomially in $m$ (cf. Section 1.2), the zeta functions $\slzeta{L}{s}$ and $\glzeta{L}{s}$ define  analytic functions on the right half-plane $\{s\in\C \mid \Re(s)>\alpha_{L}^{\SL}\}$ and $\{s\in\C \mid \Re(s)>\alpha_{L}^{\GL}\}$, where $\alpha_{L}^{\SL}$ and $\alpha_{L}^{\GL}$ denote the \emph{abscissa of convergence} of the series $\slzeta{L}{s}$ and $\glzeta{L}{s}$, respectively. From the formulas we obtained in this article, we prove:
\begin{thm}\label{thm:asymptotic}
Let $L$ be a binary $\Z$-lattice with $-4d_L$ being the  discriminant of an imaginary quadratic field.
Both $\slzeta{L}{s}$ and $\glzeta{L}{s}$ can be meromorphically continued to $\{ s \in \mathbb{C} | \Re(s) >1 \}$, and they both have a simple pole at $s=2$ in that region. In particular, we have $\alpha_L^{\SL} = \alpha_{L}^{\GL}=2$.
\end{thm}

Let $s_m^+(L):=\sum_{i=1}^{m}a_i^+(L)$ and $s_m(L):=\sum_{i=1}^{m}a_i(L)$ denote the partial sum of the number of sublattices of index at most $m$ up to (proper) isometry. Theorem \ref{thm:analytic.property} implies that the growth rate of $s_m^+(L)$ and $s_m(L)$ is of order $m^2$. To be precise, we have
\begin{align*}
s_m^+(L) & =\frac{1}{2} \mathrm{Res}_{s=2}\slzeta{L}{s} m^2 + O_{\epsilon} (m^{1+\epsilon})\\
s_m(L) &= \frac{1}{2} \mathrm{Res}_{s=2}\glzeta{L}{s} m^2 + O_{\epsilon} (m^{1+\epsilon}).
\end{align*}

We also prove the following multiplicativity result in Section \ref{sec:analytic}. For a prime $p$, let $\slzeta{L,p}{s}:=\sum_{i=0}^\infty a_{p^{i}}^+(L) p^{-is}$ denote the Dirichlet series enumerating sublattices of $p$-power index of $L$ up to proper isometry. 

\begin{thm}\label{thm:euler}
    Let $L$ be a binary $\Z$-lattice with $-4d_L$ being the  discriminant $D$ of an imaginary quadratic field. Except for the exceptional cases when $D=-3$ and $D=-4$, the zeta function $\slzeta{L}{s}$ of proper isometry classes of $L$ satisfies the Euler product    
\[\slzeta{L}{s}=\prod_{p\,prime}\slzeta{L,p}{s}\]
if and only if $\Cl_F$ is isomorphic to $(\Z/2\z)^e$ for some $e\in\N_0$.
\end{thm}

\begin{rmk}
    It is known by Chowla \cite{Chow} that there are only finitely many imaginary quadratic fields $F$ with $\Cl_F\cong (\z/2\z)^e$. Therefore there are only finitely many primitive integral binary quadratic forms $f$ up to proper equivalence such that $\slzeta{f}{s}$ has the Euler product. Apparently their form class group have exactly one class per genus.
\end{rmk}

Note that in this article, we concentrated on the case where $L$ is a binary $\z$-lattice whose discriminant $d_L$ satisfies that $-4d_L$ is the  discriminant of an imaginary quadratic field $F$. We hope to obtain results similar to Theorem \ref{thm:sl.final} and/or Theorem \ref{thm:gl.final} for any binary $\z$-lattice on a positive definite quadratic space. Furthermore, since $\slzeta{L}{s}$ and $\glzeta{L}{s}$ can be defined for any $\z$-lattice $L$, one may naturally ask whether such results can be obtained for indefinite binary lattices or even for any $n$-ary lattices.

\subsection{Organizations and methodology}
This article is arranged as follows. In Section \ref{sec:prelim}, we record preliminary results for quadratic forms and quadratic number fields, in particular regarding how to describe properties of binary lattices in terms of ideals of quadratic number fields. 

In Section \ref{sec:slcase}, we give a full description of $\slzeta{L}{s}$, proving Theorem \ref{thm:sl.final}. The function $\slzeta{L}{s}$ is represented as a finite sum of rotational terms $Z_{L,\rho}^{+}(s)$, and the calculation of $Z_{L,\rho}^{+}(s)$ becomes different depending on whether $\rho$ is identity or not. In either cases, we will heavily use the expression of $L$ in terms of ideals we established in Section \ref{sec:prelim}.

In proving Theorem \ref{thm:gl.final}, the main obstacle for calculating $\glzeta{L}{s}$ comes from evaluating reflection terms of $\glzeta{L}{s}$. In Section \ref{sec:glcase}, we first consider the rotational terms $Z_{L,\rho}(s)$ required to compute $\glzeta{L}{s}$. In Section \ref{sec:gl.refl}, we first establish properties of reflections of the binary $\z$-lattices associated to ideals of quadratic number fields. Then we consider the sum of all reflection terms ($\text{Refl}_L(s)$) at once, which is first treated in the similar way as single rotational term, but its final calculation depends on some subtle classification of ideals based on how their reflection invariant sublattices behave. We obtain Theorem \ref{thm:gl.final} by combining results from Section \ref{sec:glcase} and \ref{sec:gl.refl}.

In Section \ref{sec:examples}, we record explicit formulas for $\slzeta{L}{s}$ and $\glzeta{L}{s}$ for various binary $\z$-lattices obtained by the tools we developed in this article.  In Section \ref{sec:analytic}, we investigate analytic properties of $\slzeta{L}{s}$ and $\glzeta{L}{s}$, proving Theorem \ref{thm:asymptotic} and Theorem \ref{thm:euler}. 

\section{Preliminaries}\label{sec:prelim}

\subsection{Notations and terminologies}\label{subsec:notations}
We introduce some more notations and terminologies of quadratic spaces and lattices. Let $V$ be a positive definite quadratic space. The set of reflections of $V$ is defined by $O^-(V):=O(V)\setminus O^+(V)$. Let $L$ be a $\z$-lattice on $V$.
 The class $\cls(L)$ of $L$ is defined to be the set of all lattices on $V$ that are isometric to $L$, and the proper class $\pcls(L)$ is defined to be the set of all lattices on $V$ that are properly isometric to $L$.
The orthogonal group $O(L)$, the proper orthogonal group $O^+(L)$, and the set of reflections $O^-(L)$ of $L$ are defined by 
\[
O(L):=\{\sigma\in O(V) : \sigma(L)=L\},
\quad 
O^+(L):=O(L)\cap O^+(V) \quad \text{and} \quad O^-(L):=O(L)\setminus O^+(L).
\]
Note that $O(L)$ is finite as we are assuming $V$ is positive definite. We always assume that $Q(L)\subseteq \z$.

For $a\in\q$ and a $\z$-lattice $L$ in the quadratic space $V$, let $V^a$ denote the vector space $V$ provided with a new bilinear form $B^a(x,y):=aB(x,y)$, and we shall use $L^a$ to denote the {\em scaled lattice}, which is the lattice $L$ when it is regarded as a lattice in $V^a$.

Any unexplained notation and terminology on $\z$-lattices can be found in \cite{OM2}.
\begin{rmk}\label{rmk:zeta-inv-scaling}
    For any $a\in\q$ and a $\z$-lattice $L$, we note that $\slzeta{L}{s}=\slzeta{L^a}{s}$ and $\glzeta{L}{s}=\glzeta{L^a}{s}$. This follows since for any sublattices $K_1,K_2\subseteq L$, we have $K_1^a,K_2^a\subseteq L^a$ and $\cls(K_1)=\cls(K_2)$ (resp. $\pcls(K_1)=\pcls(K_2)$) if and only if $\cls(K_1^a)=\cls(K_2^a)$ (resp. $\pcls(K_1^a)=\pcls(K_2^a)$).
\end{rmk}

\subsection{Binary quadratic forms and ideals of imaginary quadratic fields}\label{sec:binaryform-ideal-corresp}
We focus our attention on the binary $\z$-lattices which can be identified to ideals of an imaginary quadratic number field. 

Let $F = \Q(\sqrt{d})$ be an imaginary quadratic number field with $d<0$ a square-free integer. The ring $\of{}$ of algebraic integers of $F$ is determined as
\begin{equation}\label{eqn:def-O_F}
\oo_F = \z+\z \tau, \text{ where } \tau:=\begin{cases} \frac{1+\sqrt{d}}{2} & \text{if } d \equiv 1 \Mod{4},\\
\sqrt{d} &\text{if } d \equiv 2,3 \Mod{4},\end{cases}
\end{equation}
and its discriminant $D_F:=\mathrm{Disc}(\of{})$ is given as
\begin{equation}\label{eqn:def-D_F}
D_F = \begin{cases} d & \text{if } d \equiv 1 \Mod{4}, \\ 4d & \text{if } d \equiv 2,3 \Mod{4}.\end{cases}
\end{equation}

Note that the complex conjugation $\overline{\textcolor{white}{c}}:\mathbb{C}\ra \mathbb{C}$ given by $\overline{a+bi}:=a-bi$ ($a,b\in\mathbb{R}$) 
is an automorphism of $F$ different to the identity automorphism. The norm map $N=N_{F/\q}:F\rightarrow \Q$ is given as
\[
N_{F/\q}(a+b\sqrt{d}):=(a+b\sqrt{d})(\overline{a+b\sqrt{d}})=a^2-db^2 \quad \text{for }a,b\in\Q.
\]
Note that $F$ can be viewed as a quadratic space over $\q$ with associated quadratic map $N_{F/\q}$. In this perspective, the sets $O^{+}(F)$ and $O^{-}(F)$ can be described as follows.

\begin{prop}\label{prop:oF.characterization} Let $F$ be an imaginary quadratic field, and consider $F$ as a quadratic space over $\Q$ with the norm map $N=N_{F/\q}$ as the quadratic form defined on $F$.
\begin{enumerate}[leftmargin=*, label={\rm(\arabic*)}]
\item For $\alpha\in F$, let $u_\alpha:F\ra F$ be the map defined by $u_\alpha (x)=\alpha x$. Then $O^+(F)=\{u_\alpha : N(\alpha)=1\}$.
\item $O^{-}(F) = \{\overline{\sigma} : \sigma\in O^{+}(F)\}$ where $\overline{\sigma}$ is the conjugation of $\sigma$, and all such elements have order $2$.
\item Elements $u_\alpha$ of $O^{+}(F)$ having finite order are exactly those with $\alpha$ being a unit of $\oo_F$:
\[
\of{\times} = \begin{cases}
\{1, -\omega^2, \omega, -1, \omega^2, -\omega \} & \text{if } D_F=-3, \\
\{1, i, -1, -i \} & \text{if } D_F=-4, \\
\{1, -1\} & \text{otherwise},
\end{cases}
\]
where $i=\sqrt{-1}$ and $\omega=\frac{-1+\sqrt{-3}}{2}$.
\end{enumerate}
\end{prop}

\begin{proof}
\noindent (1) Note that if $N(\alpha)=1$, then $N(\alpha x)=N(\alpha)N(x)=N(x)$ for any $x\in F$. Hence we have $u_\alpha \in O^{+}(F)$.

Now we show that any $\sigma \in O^{+}(F)$ is given as $u_\alpha$ for some $\alpha\in F$. Indeed, we claim that $\alpha=\sigma(1)$. Since $N(\sigma(x))=N(x)$ for any $x \in F$, we have $N(\sigma(1))=N(1)=1$.
Consider the map $\sigma_0:F\ra F$ defined by $\sigma_0(x) := \sigma(1)^{-1} \sigma(x)$ for any $x\in F$. Then we have $\sigma_0\in O^{+}(F)$ and  $\sigma_0(1)=1$. If we let $\sigma_0(\sqrt{d}) = a + b \sqrt{d}$ for $a,b \in \q$, then $\det(\sigma_0)=b$ so $b=1$ follows. Note that
\[
(x+ay)^2 - d y^2 = N\left((x+ay) + y \sqrt{d}\right) = N\left(\sigma_0(x + y \sqrt{d})\right) =  N\left(x+y \sqrt{d}\right) = x^2 - d y^2
\]
for any $x,y \in \q$. By comparing coefficients of the both side, we obtain $a=0$. Thus $\sigma_0(\sqrt{d})=\sqrt{d}$. From this and $\sigma_0(1)=1$, it follows that $\sigma_0$ is identity, and hence $\sigma(x) = \sigma(1)\sigma_0(x)=u_{\sigma(1)}(x)$.

\noindent (2) We first note that the conjugation map $c : F \rightarrow F$ given as $c(x) = \overline{x}$ is in $O^{-}(F)$, and that $O^{-}(F)=c \circ O^{+}(F)$. Since $(cu_\alpha)^2(x)=c(\alpha \overline{\alpha} \overline{x}) = c(\overline{x})=x$ for any $\alpha \in F$ with $N(\alpha)=\alpha \overline{\alpha}=1$, every element in $O^-(F)=c \circ O^{+}(F)$ has order $2$.

\noindent (3) If $\alpha$ has finite multiplicative order, then $\alpha$ is integral so $\alpha \in R$ should be unit of finite order. The above description of the unit group of $R$ is standard.
\end{proof}

Note that Gauss introduced the multiplication on primitive integral binary quadratic forms of the same discriminant. It is called Gauss composition law, and for each integer $D\equiv 0,1 \Mod{4}$, this induces a group structure on the set $C(D)$ of proper isometry classes of binary quadratic forms $f(x,y)=ax^2+bxy+cy^2$ with $a,b,c\in\z$, $\gcd(a,b,c)=1$ and $b^2-4ac=D$. The following is well-known:

\begin{prop}\label{prop:binary-correspond} Let $\mathcal{O}$ be the quadratic order of discriminant $D$ in an imaginary quadratic field $F$. 
\begin{enumerate}[leftmargin=*, label={\rm(\arabic*)}]
    \item If $f(x,y)=ax^2+bxy+cy^2$ is a positive definite quadratic form with $\gcd(a,b,c)=1$ and $b^2-4ac=D$, then $I=\z a + \z \left(\frac{b-\sqrt{D}}{2}\right)$ is a proper ideal of $\mathcal{O}$. The map sending $f(x,y)$ to $I$ induces an isomorphism between the form class group $C(D)$ and the ideal class group $C(\mathcal{O})$.

\item If we consider $I$ as a $\z$-lattice on the quadratic space $(F,N_{F/\q})$, then $Q_I(x,y)=a(ax^2+bxy+cy^2)$.

\end{enumerate}

\end{prop}
\begin{proof}
    For the proof of (1), we refer the readers to \cite[Theorem 7.7]{Cox}. One may easily verify that $N_{F/\q}(a)=a^2$, $N_{F/\q}\left(\frac{b-\sqrt{D}}{2}\right)=ac$, and $N_{F/\q}\left(a+\frac{b-\sqrt{D}}{2}\right)-N_{F/\q}(a)-N_{F/\q}\left(\frac{b-\sqrt{D}}{2}\right)=ab$. This proves (2).
\end{proof}
\begin{rmk}\label{rmk:zetasarethesame}
    Let $L$ be a binary $\z$-lattice such that $Q_L(x,y)=ax^2+bxy+cy^2$, $\gcd(a,b,c)=1$, and $D=b^2-4ac$. Let us consider the ideal $I$ corresponding to $Q_L$ defined as in Proposition \ref{prop:binary-correspond}. Then we have $L^a\cong I$. With Remark \ref{rmk:zeta-inv-scaling}, we thus have $\slzeta{L}{s}=\slzeta{L^a}{s}=\slzeta{I}{s}$ and $\glzeta{L}{s}=\glzeta{L^a}{s}=\glzeta{I}{s}$.
    
    Moreover, note that for any $J$ in the class group $[I]\in C(\mathcal{O})$ of $I$, we may write $J=\alpha I$ for some $\alpha\in F$. One may observe that $J\cong I^{N_{F/\q}(\alpha)}$ as $\z$-lattices on the quadratic space $F$. Hence $\slzeta{J}{s}=\slzeta{I}{s}$ and $\glzeta{J}{s}=\glzeta{I}{s}$ for any $J\in [I]$. Thus we may write $\slzeta{[I]}{s}=\slzeta{I}{s}$ and $\glzeta{[I]}{s}=\glzeta{I}{s}$ for any ideal class $[I]\in \Cl(\mathcal{O})$.
\end{rmk}

\subsection{Prime ideal splitting of quadratic number field}\label{subsec:prime.splitting}

We describe all prime ideals of $F$ by considering the splitting behavior of integral primes in $F$. As $F$ is Galois of degree $2$ over $\q$ with unique nontrivial Galois group element being the conjugation map, all possible ramification index and degree are given as follows.
    
\begin{itemize}
    \item Ramified: integral primes $r$ such that $r \vert D_F$ are ramified as $(r) = \gamma^2$ for a prime ideal $\gamma$ of $F$.
    \item Inert: integral primes $q$ such that $\left( \frac{D_F}{q} \right) =-1$ are inert in $F$.
    \item Split: integral primes $p$ such that $\left( \frac{D_F}{p} \right) =1$ are split in $F$ into two distinct prime ideals $\mathfrak{p}$ and $\overline{\mathfrak{p}}$.
\end{itemize}
\begin{rmk}Throughout this article, unless stated otherwise, we will use the letters $p,q,r$ to denote integral primes which are split, inert, ramified in $F$, respectively. Moreover, we will call the prime ideal $\gamma$ above a ramified prime the {\em ramified factor}; and call the prime ideals $\mathfrak{p}$ and $\overline{\mathfrak{p}}$ above a split prime $p$ the {\em split factors} of $p$.
\end{rmk}

The explicit description of $\Z$-basis of ramified factors $\gamma$ will be of our particular interests later on when we consider $Z_{L, \rho}(s)$ for $\rho \in O^{-}(V)$, and hence we give it as follows.

\begin{prop} \label{prop:quadIdealBasis}
Let $F=\Q(\sqrt{d})$ be an imaginary quadratic field where $d<0$ being a square-free integer. Let $r \vert D_F$ be an integral prime and let $(r)=\gamma^2$ for a prime ideal $\gamma$ in $F$. Then the $\Z$-module structure of $\gamma$ can be described as follows:
\begin{equation}\label{eqn:ramified_ideal}
    \gamma = \begin{cases}
        \Z r + \Z \frac{r+\sqrt{d}}{2} & \text{if } d\equiv 1\Mod{4},\\
        \Z 2+\Z (\sqrt{d}-1) & \text{if } d\equiv 3 \Mod{4} \text{ and } r=2,\\
        \Z r + \Z \sqrt{d} &\text{otherwise}.
    \end{cases}
\end{equation}

\end{prop}

\begin{proof}
Let $R=\of{}$. If a $\z$-lattice $M \subseteq R$ satisfies $RM \subseteq M$ and $|R:M|=r$, then $M$ is an ideal satisfying $R/M \simeq \Z/r\z$. Thus, we should have $M= \gamma$. For each case, it suffices to check that these two conditions hold for the $\z$-lattice in the right hand side of \eqref{eqn:ramified_ideal}.

If $d\equiv 1 \Mod{4}$, then $R = \Z\left[\frac{1+\sqrt{d}}{2}\right]$ and $D_F=d$ (see \eqref{eqn:def-O_F} and \eqref{eqn:def-D_F}). Observe that
    \[
\frac{1+ \sqrt{d}}{2} \cdot r = -\frac{r-1}{2} r + r  \frac{r+\sqrt{d}}{2}\quad \text{and} \quad \frac{1+ \sqrt{d}}{2} \cdot \frac{r+\sqrt{d}}{2} = \frac{d-r^2}{4r} r + \frac{r+1}{2} \frac{r+\sqrt{d}}{2}.
    \]
As $r$ is odd, $-\frac{r-1}{2}$ and $\frac{r+1}{2}$ are integers, and from $r \vert d$ and $ d \equiv 1 \Mod{4}$ it follows that $\frac{d-r^2}{4r} \in \Z$. Hence $RM \subseteq M$ for $M=\Z r + \Z \frac{r+\sqrt{d}}{2}$. Since $R = \Z 1 + \Z \frac{r+\sqrt{d}}{2}$, we easily have $|R:M| = r$.

If $d\equiv 2,3\Mod{4}$, then $R = \Z[\sqrt{d}]$ and $D_F=4d$ (see \eqref{eqn:def-O_F} and \eqref{eqn:def-D_F}). We first consider the case when $r=2$ and $2 \nmid d$. Noting that $\frac{d-1}{2}\in\z$ and that
\[
\sqrt{d}\cdot2 = 2 +2(\sqrt{d}-1)\quad \text{and} \quad \sqrt{d}\cdot(\sqrt{d}-1) = \frac{d-1}{2} 2 -(\sqrt{d}-1),
\]
we have $RM \subseteq M$ for $M=\z 2 + \z(\sqrt{d}-1)$. As $R = \Z 1 + \Z (\sqrt{d}-1)$, we have $|R:M| = 2 = r$.

Now we are left with the case when either $r$ is odd or $2 \vert d$. Noting that we always have $r\mid d$, and 
\[
\sqrt{d}\cdot r = r \sqrt{d}\quad \text{and} \quad \sqrt{d}\cdot\sqrt{d} = \frac{d}{r} \cdot r,
\]
we have $RM \subseteq M$ for $M=\Z r + \Z \sqrt{d}$. Since $R = \Z 1 + \Z \sqrt{d}$, we have $|R:M| =  r$. This completes the proof of the proposition.
\end{proof}

\subsection{A first step to enumerate $\slzeta{L}{s}$ and $\glzeta{L}{s}$} For $\rho\in O(V)$ of finite order, define 
\begin{equation}\label{eqn:def-Z+}
    Z_{L,\rho}^+(s)=\sum_{\substack{K\subseteq L \\ \rho(K)=K}}  \frac{\lvert L:K\rvert^{-s}}{\#\{\sigma\in O^+(V) : \sigma(K)\subseteq L\}}
\end{equation}
and 
\begin{equation}\label{eqn:def-Z}
    Z_{L,\rho}(s)=\sum_{\substack{K\subseteq L \\ \rho(K)=K}}  \frac{\lvert L:K\rvert^{-s}}{\#\{\sigma\in O(V) : \sigma(K)\subseteq L\}}.
\end{equation}
The following proposition reduces the computation of $\slzeta{L}{s}$ and $\glzeta{L}{s}$ to that of $Z_{L,\rho}^+(s)$ and $Z_{L,\rho}(s)$, respectively. When $\rho\in O^+(V)$, either term $Z_{L,\rho}^+(s)$ or $Z_{L,\rho}(s)$ will be called a {\em rotational term}, while each term $Z_{L,\rho}(s)$ with $\rho\in O^-(V)$ will be called a {\em reflection term}.

\begin{prop}\label{prop:zeta=sumZ}
    Let $L$ be a binary $\z$-lattice. We have 
    \[
        \slzeta{L}{s}=\sum_{\substack{\rho\in O^+(V) \\ \rho:\text{finite order}}} Z_{L,\rho}^+(s) \quad \text{and} \quad  \glzeta{L}{s}=\sum_{\substack{\rho\in O(V) \\ \rho:\text{finite order}}} Z_{L,\rho}(s)
    \]
\end{prop}
\begin{proof} Observe that 
\begin{equation}\label{eqn:slzetasplit}
\slzeta{L}{s} = \sum_{m=1}^\infty m^{-s} \sum_{K\in\mathcal{L}_m} \frac{1}{\# (\pcls(K) \cap \mathcal{L}_m)},
\end{equation}
and 
\[
\# (\pcls(K) \cap \mathcal{L}_m) = \frac{\#\{\sigma\in O^+(V) : \sigma(K)\subseteq L\}}{\#O^+(K)}.
\]
Plugging this back into \eqref{eqn:slzetasplit} gives
\begin{align*}
\slzeta{L}{s} &= \sum_{m=1}^\infty m^{-s} \sum_{K\in\mathcal{L}_m} \frac{\#O^+(K)}{\#\{\sigma\in O^+(V) : \sigma(K)\subseteq L\}}\\
&= \sum_{m=1}^\infty m^{-s} \sum_{K\in\mathcal{L}_m} \sum_{\rho\in O^+(K)} \frac{1}{\#\{\sigma\in O^+(V) : \sigma(K)\subseteq L\}}\\
&=\sum_{K\subseteq L} \frac{1}{\lvert L:K\rvert^{s}} \sum_{\rho\in O^+(K)} \frac{1}{\#\{\sigma\in O^+(V) : \sigma(K)\subseteq L\}}\\
&=\sum_{\substack{\rho\in O^+(V) \\ \rho:\text{finite order}}} \sum_{\substack{K\subseteq L \\ \rho(K)=K}}  \frac{\lvert L:K\rvert^{-s}}{\#\{\sigma\in O^+(V) : \sigma(K)\subseteq L\}} = \sum_{\substack{\rho\in O^+(V) \\ \rho:\text{finite order}}} Z_{L,\rho}^+(s).
\end{align*}
One may also obtain the claim for $\glzeta{L}{s}$ by following the same argument with a replacement of $O^+(V)$ and $O^+(K)$ by $O(V)$ and $O(K)$, respectively.
\end{proof}

\subsection{Settings for Sections \ref{sec:slcase}--\ref{sec:gl.refl}}
Here we record some general settings for Section \ref{sec:slcase} and \ref{sec:glcase}. In Section \ref{sec:slcase} (resp. Section \ref{sec:glcase}) we calculate $\slzeta{L}{s}$ (resp. $\glzeta{L}{s}$) for binary $\z$-lattices $L$ with $-4d_L$ being the discriminant of an imaginary quadratic number field $F$. 

Recall that $\slzeta{L}{s}=\slzeta{I}{s}$ and $\glzeta{L}{s}=\glzeta{I}{s}$ for the ideal $I$ of $F$ corresponding to $L$ via Proposition \ref{prop:binary-correspond}. Moreover,  $\slzeta{I}{s}$ (resp. $\glzeta{I}{s}$) only depends on the ideal class of $I$. In this manner, we may consider $Z_{L,\rho}^{+}=Z_{I,\rho}^{+}$ (resp. $Z_{L,\rho}=Z_{I,\rho}$).

For a fractional ideal $I$, we will denote the ideal class containing $I$ by $[I] \in \Cl_F$. To ease of notation, we will often omit $\Cl_F$ when we describe the set of ideal classes $[J]$ satisfying some condition: for example, the set $\{[J] : J\le \of{} , N(J)=n\}$ means $\{[J]\in \Cl_F : J\le \of{} , N(J)=n\}$, unless stated otherwise. Moreover, we will simply write the element $u_\alpha \in O^+(F)$ defined in Proposition \ref{prop:oF.characterization} as $\alpha\in O^+(F)$, identifying norm $1$ elements of $F$ with $O^+(F)$. Also, we will denote $R = \of{}$. We use the convention of writing $\subseteq$ for a $\Z$-submodule inclusion and $\le$ for an ideal ($R$-module) inclusion.

\section{Calculating zeta functions of proper isometry classes}\label{sec:slcase}
In this section, we calculate $\slzeta{L}{s}$ for binary $\z$-lattices $L$ with $-4d_L$ being the discriminant of an imaginary quadratic number field $F$. Here we always identify the quadratic space $V=\q L$ with the quadratic space $F$ with the norm map as the associated quadratic map.

Recall that
\[
\slzeta{L}{s} =\slzeta{I}{s} = \sum_{\substack{\rho\in O^+(V) \\ \rho:\text{finite order}}} Z_{I,\rho}^{+}(s).
\]
We first calculate $Z_{I,\pm 1}^{+}(s)$, what we will call the {\em trivial rotational terms}, and then calculate $Z_{I,\rho}^{+}$ when $\rho$ is a nontrivial rotation. The latter case happens only when $|\of{\times}| \neq 2$, which consists of two cases when $F = \Q(\sqrt{-1})$ or $F = \Q(\sqrt{-3})$ by Proposition \ref{prop:oF.characterization} (3).

\subsection{Calculating the trivial rotational terms $Z_{I,\pm 1}^{+}(s)$}
We first investigate the trivial rotational terms $Z_{I,\pm 1}(s)$.
\begin{prop}\label{prop:Z1+}
For any fractional ideal $I$ of $F$, we have
    \begin{equation}\label{eqn:formula-Z1+}
    Z_{I,1}^{+}(s) = Z_{I,-1}^{+}(s)=\frac{1}{|\of{\times}|}\frac{\zeta_{\Z^2}(s)}{\zeta_F(s)} \left( \sum_{n=1}^{\infty} \frac{ \# \{ [J]: J\le \of{}, \, N(J)=n\} }{n^{s}} \right).
    \end{equation}
In particular, $Z_{I,\pm1}^{+}(s)$ is only determined by $F$ and
does not depend on $I$.
\end{prop}

\begin{proof}
The first equality follows from a simple observation that in general, we have $Z_{I,\rho}^{+}(s) = Z_{I,-\rho}^{+}(s)$, since $\rho(K)=K$ and $-\rho(K)=K$ are equivalent for any $\z$-submodule $K$. So we will just consider $Z_{I,1}^{+}(s)$. Our main idea is to consider an ideal $J=RK \le I$ when counting $\z$-submodules $K \subseteq I$.

Since $\alpha K \subseteq I$ is equivalent to that $\alpha (RK) = R(\alpha K) \le I$, we have
\[
Z_{I,1}^{+}(s) = \sum_{K \subseteq I} \frac{|I : K|^{-s}}{\# \{\alpha \in O^{+}(F) : \alpha K \subseteq I \} } = \sum_{K \subseteq I} \frac{|I : K|^{-s}}{\# \{\alpha \in O^{+}(F) : \alpha (RK) \le I \} }.
\]
Putting $J = RK$, we obtain
\[
Z_{I,1}^{+}(s) = \sum_{J \le I} \frac{|I:J|^{-s}}{ \# \{\alpha \in O^{+}(F) : \alpha J \le I \} }\left( \sum_{K \subseteq J, RK=J} |J:K|^{-s}\right).
\]
For now, suppose the following holds for any fractional ideal $J$ of $R$:
\begin{equation}\label{eqn:sumideals}
    \sum_{K \subseteq J, RK=J} |J:K|^{-s} = \frac{\zeta_{\Z^2}(s)}{\zeta_F(s)}.
\end{equation}

Using \eqref{eqn:sumideals} we can proceed as
\[
Z_{I,1}^{+}(s) = \frac{\zeta_{\Z^2}(s)}{\zeta_F(s)} \sum_{J \le I} \frac{|I:J|^{-s}}{ \# \{\alpha \in O^{+}(F) : \alpha J \le I \} }.
\]
Replacing $\mfa = I^{-1}J$ for $\mfa \le R$ ideal gives
\begin{equation}\label{eqn:Z1+_intosum}
Z_{I,1}^{+}(s) = \frac{\zeta_{\Z^2}(s)}{\zeta_F(s)} \sum_{\mfa \le R} \frac{|R:\mfa|^{-s}}{ \# \{\alpha \in O^{+}(F) : \alpha \mfa \le R \} }=\frac{\zeta_{\Z^2}(s)}{\zeta_F(s)} \sum_{\mfa \le R} \frac{N(\mfa)^{-s}}{ \# \{\alpha \in O^{+}(F) : \alpha \mfa \le R \} }.
\end{equation}
It remains to evaluate
\begin{equation}\label{eqn:prop:Z1+_innersum}
\sum_{\mfa \le R} \frac{N(\mfa)^{-s}}{ \# \{ \alpha \in O^{+}(F) : \alpha \mfa \le R \} }.
\end{equation}
We consider the map
\[
\iota: \{ \alpha \in O^{+}(F) : \alpha \mfa \le R \}  \rightarrow  \{ J \le R : N(J) = N(\mfa), [\mfa]=[J] \}, \quad \iota(\alpha) = \alpha \mfa.
\]
One may easily check that $\iota$ is surjective, and $\iota(\alpha_1) = \iota(\alpha_2)$ if and only if $\alpha_1 \alpha_2^{-1} \in \of{\times}$. Thus $\iota$ is $|\of{\times}|$-to-$1$ and we have
\[
\#\{ \alpha \in O^{+}(F) : \alpha \mfa \le R \}  = |\of{\times}| \# \{ J \le R : N(J) = N(\mfa), [\mfa]=[J] \}.
\]
Plugging this into \eqref{eqn:prop:Z1+_innersum} and letting $n = N(\mfa)$, we have
\begin{equation}\label{eqn:prop:Z1+_innersum2}
\sum_{\mfa \le R} \frac{N(\mfa)^{-s}}{ \# \{ \alpha \in O^{+}(F) : \alpha \mfa \le R \} } = \frac{1}{|\of{\times}|} \sum_{n=1}^{\infty} n^{-s} \sum_{N(\mfa)=n, \mfa\le R} \frac{1}{ \# \{ J\le R : N(J)=n, [\mfa]=[J] \} }.
\end{equation}

To evaluate the inner sum on the right hand side of \eqref{eqn:prop:Z1+_innersum2}, consider the set $\{J \le R : N(J) = n\}$ and the equivalence relation $\sim$ on it given as $J_1 \sim J_2 \Leftrightarrow [J_1] = [J_2]$. Then $ \# \{ J\le R : N(J)=n, [\mfa]=[J] \}$ is size of the $\sim$-equivalence class containing $\mfa$, so adding its inverse for every $\mfa \le R$ such that $N(\mfa)=n$ counts each equivalence class exactly once. Thus we have 

\[
\sum_{N(\mfa)=n, \mfa\le R} \frac{1}{ \# \{ J\le R : N(J)=n, [\mfa]=[J] \} } = \# \{ [J] : J\le R,\, N(J)=n \}.
\]
Plugging this into \eqref{eqn:prop:Z1+_innersum2}, we have
\[
\sum_{\mfa \le R} \frac{N(\mfa)^{-s}}{ \# \{ \alpha \in O^{+}(F) : \alpha \mfa \le R \} } = \frac{1}{|\of{\times}|} \sum_{n=1}^{\infty} n^{-s} \# \{ [J]  : J\le R, \, N(J)=n \}.
\]
Substituting this to \eqref{eqn:Z1+_intosum} completes the proof of the proposition.
\end{proof}

To complete the proof of  Proposition \ref{prop:Z1+}, we need to show that Equation \eqref{eqn:sumideals} indeed holds.

\begin{lem}\label{lem:sumideals}
For any fractional ideal $I$ of $R$, we have
\[
\sum_{K \subseteq I, RK=I} |I:K|^{-s} = \frac{\zeta_{\Z^2}(s)}{\zeta_F(s)}.
\]
\end{lem}

\begin{proof}
For any fractional ideal $\mfa$ of $R$, let us define
\[
\Phi(\mfa)=\sum_{K \subseteq \mfa, RK=\mfa} |\mfa:K|^{-s}.
\]

One may easily see that if $\mfa_1$ and $\mfa_2$ are isomorphic $R$-module then $\Phi(\mfa_1) = \Phi(\mfa_2)$. Hence for $g \in \Cl_F$, we may define $\Phi(g)$ as $\Phi(g)=\Phi(\mfa)$ for any fractional ideal $\mfa$ such that $g = [\mfa]$.

We show that for any $g \in \Cl_F$ the following holds.
\begin{equation}\label{eqn:zetaz^2}
\zeta_{\Z^2}(s)= \sum_{h \in \Cl_F} \Phi(h) \left( \sum_{J_1 \le R, [J_1]=g^{-1} h} N(J_1)^{-s} \right).
\end{equation}
Taking an ideal $I\le R$ such that $[I]=g$ and considering the following alternative way of summing the submodule zeta function $\zeta_{\Z^2}(s)$, we may prove \eqref{eqn:zetaz^2}:
\begin{align*}
\zeta_{\Z^2}(s) = \sum_{K \subseteq I} |I:K|^{-s} &= \sum_{J \le I} |I : J|^{-s} \sum_{K \subseteq J, RK=J} |J:K|^{-s} \\
&= \sum_{J \le I} |I:J|^{-s} \Phi(J) \\
&= \sum_{h \in \Cl_F} \Phi(h) \left( \sum_{J \le I, [J]=h} |I:J|^{-s} \right)\\
&= \sum_{h \in \Cl_F} \Phi(h) \left( \sum_{J_1 \le R, J_1=I^{-1}J, [J_1]=[I]^{-1}h } |I:IJ_1|^{-s} \right) \\
&= \sum_{h \in \Cl_F} \Phi(h) \left( \sum_{J_1 \le R, [J_1]=g^{-1} h} N(J_1)^{-s} \right).
\end{align*}

Now we would like to express the following inner sum of \eqref{eqn:zetaz^2} using Hecke L-functions:
\[
\left( \sum_{J_1 \le R, [J_1]=g^{-1} h} N(J_1)^{-s} \right).
\]
Any group character $\chi \in (\Cl_F)^{\vee}$, or equivalently a group homomorphism $\Cl_F \rightarrow \{ z \in \C : |z|=1 \}$, can be interpreted as a Hecke character of conductor 1, and its Hecke L-function can be described as
\[
L(s, \chi) = \sum_{J \le R} \chi(J) N(J)^{-s} = \sum_{g \in \Cl_F} \chi(g) \left( \sum_{J \le R, [J]=g} N(J)^{-s} \right).
\]
By the orthogonality of characters, for any $f\in \Cl_F$ we have 
\begin{align*}
\sum_{\chi \in (\Cl_F)^{\vee}} \chi(f)^{-1} L(s,\chi) &= \sum_{\chi \in (\Cl_F)^{\vee}} \sum_{g \in \Cl_F} \chi(f)^{-1} \chi(g) \left( \sum_{J \le R, [J]=g} N(J)^{-s} \right) \\
&= |\Cl_F| \sum_{g \in \Cl_F} \delta_{fg}  \left( \sum_{J \le R, [J]=g} N(J)^{-s} \right) \\
&= |\Cl_F| \sum_{J \le R, [J]=f} N(J)^{-s},
\end{align*}
where $\delta_{fg}=1$ if $f=g$ and $\delta_{fg}=0$ otherwise.

Plugging this into \eqref{eqn:zetaz^2}, we have
\[
\zeta_{\Z^2}(s)= \frac{1}{|\Cl_F|} \sum_{h \in \Cl_F} \Phi(h) \sum_{\chi \in (\Cl_F)^{\vee}} \chi(g) \chi(h)^{-1} L(s,\chi)
\]
for any $g \in \Cl_F$. For a character $\psi \in (\Cl_F)^{\vee}$, multiplying the above by $\psi(g)^{-1}$ and then adding over all $g\in \Cl_F$, we obtain
\begin{equation}\label{eqn:zetasums}
\begin{aligned}
\left( \sum_{g \in \Cl_F} \psi(g)^{-1} \right) \zeta_{\Z^2}(s) &= \frac{1}{|\Cl_F|} \sum_{g, h \in \Cl_F, \chi \in (\Cl_F)^{\vee}} \Phi(h) \psi(g)^{-1} \chi(g) \chi(h)^{-1} L(s,\chi) \\
&= \frac{1}{|\Cl_F|} \sum_{h \in \Cl_F, \chi \in (\Cl_F)^{\vee}} \Phi(h)  \chi(h)^{-1} L(s, \chi) \sum_{g \in \Cl_F} \psi(g)^{-1} \chi(g) \\
&= \sum_{h \in \Cl_F} \Phi(h) \psi(h)^{-1} L(s, \psi). 
\end{aligned}
\end{equation}
Meanwhile, the left hand side of \eqref{eqn:zetasums} is given as
\[
\left( \sum_{g \in \Cl_F} \psi(g)^{-1} \right) \zeta_{\Z^2}(s) = \begin{cases} |\Cl_F|\zeta_{\Z^2}(s) & \text{if } \psi = \chi_0, \\ 0 & \text{otherwise},
\end{cases}
\]
where $\chi_0$ is the trivial character. Dividing both sides of \eqref{eqn:zetasums} by $L(s,\psi)$ and noting that $L(s, \chi_0) = \zeta_F(s)$, which follows  from the definition of $\zeta_F(s)$, we have
\begin{equation}\label{eqn:Phicharactersum}
\sum_{h \in \Cl_F} \Phi(h) \psi(h)^{-1} = \begin{cases} |\Cl_F|\frac{\zeta_{\Z^2}(s)}{\zeta_F(s)} & \text{if } \psi= \chi_0 \\ 0 & \text{otherwise}.
\end{cases}
\end{equation}
Finally, for $g \in \Cl_F$, multiplying \eqref{eqn:Phicharactersum} by $\psi(g)$ and adding over all $\psi\in (\Cl_F)^\vee$, we have
\begin{align*}
|\Cl_F| \Phi(g) = |\Cl_F| \sum_{h \in \Cl_F} \delta_{gh} \Phi(h) &= \sum_{\psi \in (\Cl_F)^{\vee}} \sum_{ h \in \Cl_F} \psi(g) \psi(h)^{-1} \Phi(h) \\
&= \sum_{\psi \in (\Cl_F)^{\vee}}\psi(g)  \sum_{ h \in \Cl_F} \psi(h)^{-1} \Phi(h)= |\Cl_F| \chi_0(g) \frac{\zeta_{\Z^2}(s)} {\zeta_F(s)} = |\Cl_F| \frac{\zeta_{\Z^2}(s)} {\zeta_F(s)}.
\end{align*}
This completes the proof of the lemma.
\end{proof}

\subsection{Calculating $Z_{I,\rho}^{+}(s)$ for $\rho \neq \pm 1$}

We checked in Proposition \ref{prop:oF.characterization} that rotations $\rho \in O^{+}(F)$ having finite order are exactly elements of $\of{\times}$, and such rotations other than $1$ or $-1$ exist only when $F=\Q(\sqrt{-1})$ or $F = \Q(\sqrt{-3})$. For such cases, $R = \oo_F$ is given as $R = \Z[i]$ and $R = \Z[\omega]$ for $i = \sqrt{-1}$ and $\omega = \frac{-1+ \sqrt{-3}}{2}$ respectively, and both of those rings are PID. As $Z^{+}_{I,\rho}$ only depends on the $R$-module structure of $I$, we have $Z^{+}_{I,\rho} = Z^{+}_{R,\rho}$ when $R$ is a PID and thus all fractional ideals are isomorphic as $R$-module. We will calculate $Z^{+}_{R, \rho}$ when $R = \Z[i]$ or $\Z[\omega]$ and $\rho \in R^{\times} \setminus \{1,-1\}.$

\begin{prop} \label{prop:Z+_rotation} We have
\[
Z_{\Z[i],i}^{+}(s)=Z_{\Z[i],-i}^{+}(s) = \frac{1}{4} \zeta_{\Q(\sqrt{-1})}(s) \prod_{p \equiv 1 \Mod{4}} (1-p^{-s})
\]
and for $k \in \{1,2\}$, we have
\[
Z_{\Z[\omega],\omega^k}^{+}(s) = Z_{\Z[\omega],-\omega^k}^{+}(s) = \frac{1}{6} \zeta_{\Q(\sqrt{-3})}(s) \prod_{p \equiv 1 \Mod{3}} (1-p^{-s}).
\]
\end{prop}

\begin{proof}
We will prove the following statement: for $R =\of{}= \Z[i]$ or $\Z[\omega]$ and $\rho \in \of{\times} \setminus \{1, -1\}$, we have
\[
Z_{R, \rho}^{+}(s) = \frac{1}{|\of{\times}|} \sum_{n \in \mathcal{N}(R)} n^{-s}
\]
where $\mathcal{N}(R) = \{N(I) : I \le R \}$.

Noting that $R=\z[\rho]$, any $\z$-submodule $K\subseteq R$ with $\rho(K)=K$ is an $R$-ideal of $R$. Thus we have
\[
Z_{R,\rho}^{+}(s) =\sum_{\substack{K\subseteq R \\ \rho(K)=K}}  \frac{\lvert R:K\rvert^{-s}}{\#\{\sigma\in O^+(F) : \sigma(K)\subseteq R\}}=\sum_{I \le R} \frac{|R : I|^{-s}}{\# \{\alpha \in O^{+}(F) : \alpha I \le R \} },
\]
and we already evaluated it in the proof of Proposition \ref{prop:Z1+} as
\[
\frac{1}{|\of{\times}|}  \left( \sum_{n=1}^{\infty} \frac{ \# \{ [I]: N(I)=n\} }{n^{s}} \right).
\]
As $R$ is a PID, all ideal classes are trivial, so
\[
\left\{ [I]: N(I)=n \right\}= \begin{cases} 1 & \text{if } n \in \mathcal{N}(R), \\ 0 &\text{otherwise}. \end{cases}
\]
Thus we have
\[
Z_{R, \rho}^{+}(s) = \frac{1}{|\of{\times}|} \sum_{n \in \mathcal{N}(R)} n^{-s}.
\]
To evaluate this sum over $\mathcal{N}(R)$, note that $\mathcal{N}(R)$ is multiplicatively generated by norms of all prime ideals of $R$, so it is generated by numbers of form $p, q^2, r$ for $p, q, r$ primes which are split, inert, ramified primes in $F$ respectively. This shows that
\begin{equation}\label{eqn:sumover-mathcal(N)(R)}
\sum_{n \in \mathcal{N}(R)} n^{-s} = \prod_{p \text{ split}}(1-p^{-s})^{-1} \prod_{q \text{ inert}} (1-q^{-2s})^{-1} \prod_{r \text{ ramified}} (1-r^{-s})^{-1}= \zeta_F(s) \prod_{p\text{ split}} (1-p^{-s}).
\end{equation}
Noting that for $R=\Z[i]$ or $\Z[\omega]$, the split primes are given as primes $p \equiv 1 \Mod{4}$ and $p \equiv 1 \Mod{3}$ respectively, and that $|\of{\times}|=4$ and $|\of{\times}|=6$ respectively, we may complete the proof.
\end{proof}

Combining Propositions \ref{prop:zeta=sumZ}, \ref{prop:Z1+} and \ref{prop:Z+_rotation}, we may obtain the following formula of $\slzeta{L}{s}$.

\begin{thm} \label{thm:slTotal} Let $L$ be a binary $\z$-lattice such that $-4d_L$ is the discriminant of the imaginary quadratic field $F=\q\left(\sqrt{-4d_L}\right)$. Then we have
\[
\slzeta{L}{s} = \frac{2}{|\of{\times}|} \frac{\zeta_{\Z^2}(s)}{\zeta_{F}(s)} \left( \sum_{n=1}^{\infty} \frac{ \# \{ [J]: J \le \of{}, N(J)=n\} }{n^{s}} \right) +  \frac{|\of{\times}| - 2}{|\of{\times}|}  \prod_{p \text{ split}} (1-p^{-s}) \zeta_{F}(s).
\]
\end{thm}

\begin{proof}
Let $I$ be a fractional ideal corresponding to $L$ described in Proposition \ref{prop:binary-correspond}. By Remark \ref{rmk:zetasarethesame} and Proposition \ref{prop:zeta=sumZ}, and by identifying $O^+(F)$ with $\of{\times}$ we have
\[
\slzeta{L}{s}=\slzeta{I}{s} = Z_{I,1}^{+}(s) + Z_{I,-1}^{+}(s) + \sum_{\rho \in \of{\times} \setminus \{1, -1\}}
 Z_{I,\rho}^{+}(s).
\]
By Proposition \ref{prop:Z1+}, we have
\[
Z_{I,1}^{+}(s) + Z_{I,-1}^{+}(s) = \frac{2}{|\of{\times}|} \frac{\zeta_{\Z^2}(s)}{\zeta_{F}(s)} \left( \sum_{n=1}^{\infty} \frac{ \# \{ [J] : J \le \of{} , N(J)=n\} }{n^{s}} \right).
\]
Hence if $|\of{\times}|=2$, then we are done.
If $|\of{\times}|>2$, then the number of $\rho\in \of{\times}\setminus\{1,-1\}$ is $|\of{\times}|-2$, and since Proposition \ref{prop:Z+_rotation} shows that
\[
Z_{I,\rho}^+(s) = \frac{1}{|\of{\times}|} \prod_{p \text{ split}}(1-p^{-s}) \zeta_F(s),
\]
the result follows. This completes the proof of the theorem.
\end{proof}

\section{Calculating zeta functions of isometry classes -- the rotational terms}\label{sec:glcase}

We will next calculate $\glzeta{L}{s}$ for binary $\z$-lattices $L$ with $-4d_L$ being the discriminant of an imaginary quadratic number field $F$. 
By Proposition \ref{prop:oF.characterization} (2) and Proposition \ref{prop:zeta=sumZ}, we have
\[
\glzeta{L}{s} =\glzeta{I}{s}= \mathrm{Rot}_I(s) + \mathrm{Refl}_I(s),
\]
where 
\begin{equation}\label{eqn:def-Rot-Refl}
\mathrm{Rot}_I(s):=\sum_{\substack{\rho\in O^+(V) \\ \rho:\text{finite order}}} Z_{I,\rho}(s) \quad \text{and} \quad \mathrm{Refl}_I(s):= \sum_{\rho \in O^{-}(V)} Z_{I,\rho}(s).
\end{equation}
We will calculate $\mathrm{Rot}_I(s)$ in this section, and calculate $\mathrm{Refl}_I(s)$ in Section \ref{sec:gl.refl}. We call $\mathrm{Rot}_I(s)$ and $\mathrm{Refl}_I(s)$ as the sum of all rotational/reflection terms for $\glzeta{L}{s}$ respectively.

The sum $\mathrm{Rot}_I(s)$ consists of finite number of $Z_{I,\rho}$, and it only consists of $Z_{I, \pm 1}$ unless $F = \Q(\sqrt{-1})$ or $F = \Q(\sqrt{-3})$. Calculation of $Z_{I,\pm 1}$ is quite similar to that of $Z_{I, \pm 1}^{+}$ appeared in Proposition \ref{prop:Z1+}, but there is some subtle difference arising from $O^{-}(V)$-equivalence of lattices. Other nontrivial rotational terms $Z_{I,\rho}$, if exist, turn out to be identical to $Z_{I,\rho}^{+}$.

\begin{prop} \label{prop:Z1}
For a fractional ideal $I$ of $F$, we have
\begin{align*}
Z_{I,1}(s) &= \frac{1}{2} Z_{I,1}^+(s) + \frac{\zeta_{\Z^2}(s)}{\zeta_{F}(s)} \frac{1}{2|\of{\times}|}\sum_{n=1}^{\infty} \frac{\# \left( \{ [J] : |I : J|=n \} \setminus \{ [\overline{J}] : |I : J|=n \} \right)}{n^s}\\
&= \frac{1}{2|\of{\times}|}\frac{\zeta_{\Z^2}(s)}{\zeta_{F}(s)}\left(\sum_{n=1}^{\infty} \frac{\# \{ [J] : |I : J|=n \}}{n^s}+ \sum_{n=1}^{\infty} \frac{\# \left( \{ [J] : |I : J|=n \} \setminus \{ [\overline{J}] : |I : J|=n \} \right)}{n^s} \right).
\end{align*}
\end{prop}

\begin{proof}\phantom{\qedhere}
As in Proposition \ref{prop:oF.characterization}, we may identify $O^+(F)$ with norm $1$ elements of $F$, and then $O^{-}(F)$ equals $O^{+}(F)$ composed with the complex conjugation. Hence for a $\z$-submodule $K \subseteq I$ we have
\[
\# \{ \sigma \in O(F) : \sigma(K) \subseteq I \} = \# \{\alpha \in O^{+}(F) : \alpha K \subseteq I \} + \# \{\alpha \in O^{+}(F) : \overline{ (\alpha K)} \subseteq I \}.
\]
Plugging this into \eqref{eqn:def-Z}, we have for any $\rho\in O(F)$ that
\begin{equation}\label{eqn:Z_I,rho}
Z_{I,\rho}(s) = \sum_{K \subseteq I, \, \rho(K)=K} \frac{|I : K|^{-s}}{\# \{\alpha \in O^{+}(F) : \alpha K \subseteq I \} + \# \{\alpha \in O^{+}(F) : \overline{ (\alpha K)} \subseteq I \} }.
\end{equation}
For ideal $J = RK$ and $\alpha \in O^{+}(F)$, we have $\alpha K \subseteq I$ if and only if $\alpha J \le I$; and
\[
\overline{ (\alpha K) } \subseteq I \Leftrightarrow \alpha K  \subseteq \overline{I} \Leftrightarrow \alpha J \le \overline{I} \Leftrightarrow \overline{ (\alpha J)} \le I.
\]
Hence we may write
\begin{align}
Z_{I,1}(s) &= \sum_{J \le I} \frac{|I : J|^{-s}}{\# \{\alpha \in O^{+}(F) : \alpha J \le I \} + \# \{\alpha \in O^{+}(F) : \overline{ (\alpha J)} \le I \} } \left( \sum_{K \subseteq J, RK =J}|J:K|^{-s} \right) \label{eqn:prop:Z1-intermediate-sum1} \\
&=\frac{\zeta_{\Z^2}(s)}{\zeta_{F}(s)} \sum_{J \le I} \frac{|I : J|^{-s}}{\# \{\alpha \in O^{+}(F) : \alpha J \le I \} + \# \{\alpha \in O^{+}(F) : \overline{ (\alpha J)} \le I \} }. \label{eqn:prop:Z1-intermediate-sum2}
\end{align}
Here we used Lemma \ref{lem:sumideals} to simplify the inner sum in \eqref{eqn:prop:Z1-intermediate-sum1}. Meanwhile, we have
\begin{align}
\# \{\alpha \in O^{+}(F) : \alpha J \le I \} &= |\of{\times}| \# \{J_1 \le I : |I:J_1| = |I : J|, [J_1]=[J] \}, \label{eqn:counting-oftimes-to-one-1}\\
\# \{\alpha \in O^{+}(F) : \overline{ (\alpha J)} \le I \} &= |\of{\times}|  \# \{J_1 \le I : |I:J_1| = |I : J|, [J_1]=[\overline{J}] \} \label{eqn:counting-oftimes-to-one-2}
\end{align}
by using similar argument used in the proof of Proposition \ref{prop:Z1+}.

We now interpret the sum in \eqref{eqn:prop:Z1-intermediate-sum2} as the number of certain equivalence classes on the set 
\[
S_{I, n} = \{ (J, \omega_J) : J \le I, |I : J| = n, \omega_J \in \{-1,1\} \}
\]
We consider an equivalence relation $\sim_1$ on $S_{I,n}$ defined as follows: two elements $(J_1, \omega_1)$ and $(J_2, \omega_2)$ are equivalent if and only if either $J_2 = \alpha J_1$ and $\omega_1 = \omega_2$ for some $\alpha \in O^{+}(F)$; or $J_2 = \overline{\alpha J_1}$ and $\omega_1 = -\omega_2$ for some $\alpha \in O^{+}(F)$. One may easily check that it is a well-defined equivalence relation, and 
\[
\# \{\alpha \in O^{+}(F) : \alpha J \le I \} + \# \{\alpha \in O^{+}(F) : \overline{ (\alpha J)} \le I \}
\]
equals to $|\of{\times}|$ times the size of the $\sim_1$-equivalence class of $S_{I, n}$ containing $(J,\omega_J)$. By \eqref{eqn:counting-oftimes-to-one-1} and \eqref{eqn:counting-oftimes-to-one-2}, we have
\begin{equation}\label{eqn:prop:Z1-sumintoequvclasses}
\sum_{J \le I, |I : J|=n} \frac{1}{\# \{\alpha \in O^{+}(F) : \alpha J \le I \} + \# \{\alpha \in O^{+}(F) : \overline{ (\alpha J)} \le I \} }  = \frac{1}{2 |\of{\times}|} \# \left( S_{I,n} /_{\sim_1} \right).
\end{equation}

On the other hand, let us consider another equivalence relation $\sim_2$ on $S_{I,n}$ defined as $(J_1, \omega_1) \sim_2 (J_2, \omega_2)$ if and only if $[J_1] = [J_2]$ and $\omega_1 = \omega_2$. Note that the $\sim_2$-equivalence implies the $\sim_1$-equivalence. 
The $\sim_2$-equivalence classes $S_{I,n} /_{\sim_2}$ of $S_{I,n}$ can be identified as
\[
\{ ([J], \omega_J) : J \le I, |I:J|=n, \omega_J \in  \{-1,1\} \}.
\]
Applying $\sim_1$-equivalence again in $S_{I,n} /_{\sim_2}$, the $\sim_1$-equivalence classes of $S_{I,n} /_{\sim_2}$ is either of form the $\{ ([J], \omega_J), ([\overline{J}], -\omega_J) \}$ or $\{ ([J], \omega_J)\}$. Then the number of $\sim_1$-equivalence is given as
\begin{equation}\label{eqn:num_equivclasses1}
\# (S_{I, n} /_\sim) = \frac{1}{2} \left( \# (S_{I, n} /_{\sim_2}) + \#\{ ([J], \omega_J)  \in S_{I, n} : ([\overline{J}], -\omega_J) \notin S_{I, n} \} \right).
\end{equation}

We can eliminate descriptions of an element $\omega_J \in  \{-1,1\}$ in both terms, that is, we have
\begin{equation}\label{eqn:num_equivclasses2}
\frac{1}{2} \# (S_{I, n} /_{\sim_2}) = \# \{ [J] : |I : J|=n \} = \{ [J] : J \le R, N(J)=n \}    
\end{equation}
and
\begin{equation}\label{eqn:num_equivclasses3}
\frac{1}{2} \#\{ ([J], \omega_J)  \in S_{I, n} : ([\overline{J}], -\omega_J) \notin S_{I, n} \} 
= \# \left( \{ [J] : |I : J|=n \} \setminus \{ [\overline{J}] : |I : J|=n \} \right).
\end{equation}
Plugging \eqref{eqn:prop:Z1-sumintoequvclasses} into \eqref{eqn:prop:Z1-intermediate-sum2} and the plugging in the equations \eqref{eqn:num_equivclasses1}, \eqref{eqn:num_equivclasses2}, \eqref{eqn:num_equivclasses3} for $\# (S_{I,n}/_{\sim_1})$, we have
\begin{align*}
Z_{I,1}(s) &= \frac{\zeta_{\Z^2}(s)}{\zeta_{F}(s)} \frac{1}{2|\of{\times}|} \sum_{n=1}^{\infty} \frac{\# \left( S_{I,n} /_{\sim_1} \right)}{n^s} \\
&= \frac{\zeta_{\Z^2}(s)}{\zeta_{F}(s)} \frac{1}{2|\of{\times}|} \sum_{n=1}^{\infty} \frac{\# \{[J] : J \le R, N(J)=n \}}{n^s} \\
&+\frac{\zeta_{\Z^2}(s)}{\zeta_{F}(s)} \frac{1}{2|\of{\times}|} \sum_{n=1}^{\infty} \frac{\# \left( \{ [J] : |I : J|=n \} \setminus \{ [\overline{J}] : |I : J|=n \} \right)}{n^s}.\tag*{\qed}
\end{align*}
\end{proof}

Now we derive formulas for nontrivial rotational terms $Z_{I,\rho}$.

\begin{prop} \label{prop:Z_rotation} For $F=\Q(\sqrt{-1})$ or $F=\Q(\sqrt{-3})$ and $\rho \in O^{+}(F)$, we have $Z_{I,\rho}(s) = \frac{1}{2} Z_{R,\rho}^{+}(s)$. \end{prop}

\begin{proof}\phantom{\qedhere}
In both cases $R$ is a PID, so $Z_{I,\rho}(s) = Z_{R,\rho}(s)$. For $K \subseteq R$, two conditions $\alpha K \subseteq R$ and $\overline{\alpha K} \subseteq R$ for $\alpha \in O^{+}(F)$ are equivalent, so we have
\begin{align*}
Z_{R,\rho}(s) &= \sum_{K \subseteq R, \rho(K) = K} \frac{|R:K|^{-s}}{\# \{ \alpha \in O^{+}(F) : \alpha K \subseteq R \}+ \# \{ \alpha \in O^{+}(F) : \overline{\alpha K} \subseteq R \}} \\
&= \sum_{K \subseteq R, \rho(K) = K} \frac{|R:K|^{-s}}{2 \# \{ \alpha \in O^{+}(F) : \alpha K \subseteq R \}}\\
&=  \frac{1}{2} Z_{R,\rho}^{+}(s).\tag*{\qed}
\end{align*}
\end{proof}

Combining Propositions \ref{prop:Z1} and \ref{prop:Z_rotation}, we provide a formula for the sum $\mathrm{Rot}_I(s)$ of all rotational terms.
\begin{cor}\label{cor:glzeta-rotationterm} 
For any ideal $I \le \oo_F$, the sum $\mathrm{Rot}_I(s)$ of all rotational term is given as
\[
\mathrm{Rot}_I(s) = \frac{1}{2} \slzeta{I}{s} + \frac{\zeta_{\Z^2}(s)}{2\zeta_{F}(s)} \sum_{n=1}^{\infty} \frac{\# \left( \{ [J] : |I : J|=n \} \setminus \{ [\overline{J}] : |I : J|=n \} \right)}{n^s}.
\]
\end{cor}
\begin{proof}\phantom{\qedhere}
If $|\of{\times}|=2$, then $\rho \in \{1,-1\}$. Hence we have from Propositions \ref{prop:zeta=sumZ} and \ref{prop:Z1} that
\begin{align*}
\mathrm{Rot}_I(s) &= Z_{I,1}(s) + Z_{I,-1}(s) \\
=& \frac{1}{2} Z_{I,1}^{+}(s) + \frac{1}{2} Z_{I,-1}^{+}(s) + \frac{\zeta_{\Z^2}(s)}{\zeta_{F}(s)} \frac{1}{|\of{\times}|}\sum_{n=1}^{\infty} \frac{\# \left( \{ [J] : |I : J|=n \} \setminus \{ [\overline{J}] : |I : J|=n \} \right)}{n^s} \\
=& \frac{1}{2} \slzeta{I}{s} + \frac{\zeta_{\Z^2}(s)}{2\zeta_{F}(s)} \sum_{n=1}^{\infty} \frac{\# \left( \{ [J] : |I : J|=n \} \setminus \{ [\overline{J}] : |I : J|=n \} \right)}{n^s}.
\end{align*}
If $|\of{\times}|>2$, then Propositions \ref{prop:zeta=sumZ} and \ref{prop:Z_rotation} show that 
\[
\mathrm{Rot}_I(s) =\sum_{\rho \in O^{+}(F),\text{finite order}} \frac{1}{2} Z_{I,\rho}^{+}(s) = \frac{1}{2} \slzeta{I}{s}.
\]
On the other hand, since $R$ is PID in this case, we always have $[J]=[\overline{J}]$. Hence
\[
\sum_{n=1}^{\infty} \frac{ \# \left( \{ [J] : |I : J|=n \} \setminus \{ [\overline{J}] : |I : J|=n \} \right)}{n^s}=0.
\tag*{\qed}\]
\end{proof}

\section{Calculating zeta functions of isometry classes -- the reflection terms}\label{sec:gl.refl}

In this section we complete our evaluation of $\glzeta{L}{s}$ by calculating  the sum of all reflection terms
\[
\mathrm{Refl}_I(s) = \sum_{\rho \in O^{-}(F)} Z_{I,\rho}(s).
\]
While there are infinitely many $Z_{I,\rho}$ involved here, it turns out that the sum $\mathrm{Refl}_I(s)$ as a single entity exhibits analogous behavior to $Z_{I,1}$, and it can be approached by the similar strategy to the proof of Proposition \ref{prop:Z1+} at the beginning. But the auxiliary terms $\Psi(I)$ (defined in Section 5.2) appearing in the process are much more complicated than their $\slzeta{L}{s}$-counterparts $\Phi(I)$ of Lemma \ref{lem:sumideals}, and calculating them requires deep understanding of how ideals of $R=\of{}$ and their sublattices behave with respect to reflections.

We call a $\z$-lattice on a quadratic space $V$ {\it reflexive} if it is invariant under an element of $O^{-}(V)$. We will first discuss how to characterize reflexive $\z$-lattices on $F$ in terms of number theoretic descriptions, and how to count reflection invariant sublattices of a given reflexive lattice. After then we will be able to rearrange the sum $\mathrm{Refl}_I(s)$ in a similar fashion to $Z_{I,1}$ and evaluate the occurring auxiliary terms $\Psi(I)$ in Proposition \ref{prop:E[I]}.

\subsection{Preliminaries} We start by characterizing reflections of $F$ viewed as quadratic spaces.

\begin{prop}\label{prop:reflection}
Any reflection in $O^{-}(F)$ takes the form of
\[
P_v(w) = \frac{v}{ \overline{v} } \overline{w}
\]
for some element $v \in R$.
\end{prop}

\begin{proof}
Note that the composition $\rho\circ c$ of the complex conjugation $c$ and $\rho \in O^{-}(F)$ is in $O^{+}(F)$. Hence $\rho(\overline{w})= \alpha w$ for some $\alpha \in O^{+}(F)$. By Hilbert's theorem 90 there exists $v \in F$ such that $\alpha = v/\overline{v}$, and we may assume that $v \in R$ by multiplying some integer to $v$, if necessary. Then it follows $\rho(w) = \alpha \overline{w}= (v/\overline{v}) \overline{w}$.
\end{proof}

The map $P_v$ is the reflection which preserves the line $\Q v$ and flips the line $\Q \left(\sqrt{d}v\right)$ which is orthogonal to $\Q v$. It follows that a $\z$-lattice on $F$ is reflexive if and only if it is invariant under $P_v$ for some $v\in R$.

\begin{prop}\label{prop:reflection-property}
Let $K\subseteq R$ be a $\z$-lattice on $F$. We have the following.
\begin{enumerate}[label={\rm(\arabic*)}]
\item If $P_v K = K$, then $P_v(RK) = RK$.

\item If $P_v K = K$, then $P_{\sqrt{d} v} K = K$.

\item A fractional ideal $I$ is reflexive if and only if $[I] = [\overline{I}]$ if and only if $[I]^{-1}=[I]$.

\item For a reflexive fractional ideal $I$, $O^{-}(I)$ is a $R^{\times}$-coset. In particular, $|O^{-}(I)| = |R^{\times}|$.

\item Suppose that $|R^{\times}|=2$. If $P_v K = K$, then $O^{-}(K)=O^{-}(RK) = \{P_v, P_{\sqrt{d} v}\}. $ In particular, $|O^{-}(K)|=2$ for any reflexive $\z$-lattice $K$ on $F$.
\end{enumerate}
\end{prop}

\begin{proof}
(1) For $r \in R$ and $k \in K$, we have
\[
P_v(rk) = \frac{v}{\overline{v}} \overline{r}\overline{k} = \overline{r} P_v(K).
\]
Since $\overline{r} \in R$ and $P_v(K) \in K$, we have $P_v(rk) \in RK$. \\
(2) One may check that $P_v P_{\sqrt{d} v}x  = -x$ always holds using Proposition \ref{prop:reflection}. \\
(3) Since $I\overline{I}=(N(I))$ in general (see \cite[Lemma 7.14]{Cox}), we have $[I] = [\overline{I}]$ if and only if $[I]^{-1}=[I]$. To prove the first equivalence, note that Proposition \ref{prop:reflection} shows that $P_v I = (v / \overline{v}) \overline{I}$. Hence if $P_v I = I$, then $(v / \overline{v}) \overline{I} = I$, hence $[I] = [\overline{I}]$. Conversely, if $[I] = [\overline{I}]$, then there exists $\alpha \in F^{\times}$ such that $\alpha \overline{I} = I$. The norm of such $\alpha$ should be $1$, so by Hilbert's theorem 90 there exists $v \in R$ such that $\alpha = v / \overline{v}$. Thus we have $P_v I = (v / \overline{v}) \overline{I} = I$. \\
(4) Since $I$ is reflexive, $O^-(I)$ is nonempty, and hence $O^-(I)$ is an $O^+(I)$-coset. Thus it suffices to show that  $O^{+}(I) = R^{\times}$. Since $I$ is closed in $R$-multiplication, we have $R^{\times} \subseteq O^{+}(I)$. Since any element in $O^{+}(I)$ should have finite order, we have $O^{+}(I) \subseteq R^{\times}$ by Proposition \ref{prop:oF.characterization} (3). \\
(5) By part (2) we have $\{P_v, P_{\sqrt{d} v} \} \subseteq O^{-}(K)$, and by part (1) we have $O^{-}(K) \subseteq O^{-}(RK)$. Meanwhile part (4) shows that  $|O^{-}(RK)| =|R^{\times}|=2$. Thus we may conclude that $O^{-}(K)=O^{-}(RK) = \{P_v, P_{\sqrt{d} v}\}$.
\end{proof}

\begin{prop}\label{prop:reflexive-TFAE}
For reflexive fractional ideals $I, J$ of $F$, the following are equivalent.
\begin{enumerate}[label={\rm(\arabic*)}]
\item $O^{-}(I) = O^{-}(J)$.
\item $O^{-}(I) \cap O^{-}(J) \neq \emptyset$.
\item $\overline{I} I^{-1} = \overline{J} J^{-1}$.
\item $IJ^{-1}$ is a product of powers of ramified factors and a rational number.
\end{enumerate}
\end{prop}

\begin{proof}
$(1) \Leftrightarrow (2)$: This follows from Proposition \ref{prop:reflection-property} (4) -- two $R^\times$ cosets are either the same or disjoint. \\
$(2) \Rightarrow (3)$: Let $P_v\in O^{-}(I) \cap O^{-}(J)$. By Proposition \ref{prop:reflection}, $P_v I= I$ if and only if $\overline{I}I^{-1} = \overline{v} v^{-1}$. Thus $P_v I=I$ and $P_v J = J$ implies that $\overline{I}I^{-1} = \overline{J}J^{-1}$.\\
$(3) \Rightarrow (2)$: Since $[I] = [\overline{I}]$ by Proposition \ref{prop:reflection-property} (3), the ideal $\overline{I}I^{-1} = \overline{J}J^{-1}$ is principal, and is of norm $1$. By Hilbert's theorem 90 there exists $ v \in R$ such that $I \overline{I}^{-1} = v/\overline{v}$. Hence $P_v I = I$ and $P_v J = J$.\\
$(3) \Leftrightarrow (4)$: Putting $M=IJ^{-1}$, the condition $\overline{I} I^{-1} = \overline{J} J^{-1}$ is equivalent to $M = \overline{M}$. We consider the prime ideal factorization of $M$.
Note that the complex conjugation action in $F$ preserves all inert primes and ramified factors, and swaps two split factors $\mathfrak{p}_1$ and $\mathfrak{p}_2$ of a split prime $p = \mathfrak{p}_1 \mathfrak{p}_2$.
Hence $M = \overline{M}$ if and only if $M$ contains the same power of two split factors of $p$ for each of the split primes $p$.
Equivalently, $M$ is a product of powers of ramified factors and a rational number.
\end{proof}

For a reflexive $\z$-lattice $L$ on $F$ and $\rho \in O^{-}(L)$, we say $L$ is {\em orthogonal with respect to $\rho$} if there exists a $\Z$-basis $\{v,w\}$ of $L$ such that $\rho(v)=v$ and $\rho(w) = -w$. We say $L$ is {\em orthogonal} if $L$ is orthogonal with respect to a $\rho\in O^-(L)$. Otherwise, we call $L$ {\em nonorthogonal}. 
Moreover, we call a fractional ideal $I$ an {\em (non)orthogonal ideal} if it is (non)orthogonal as a $\z$-lattice on $F$.
We will see that distinguishing those two types of reflexive lattices becomes crucial in counting reflexive sublattices.

\begin{prop}\label{prop:P_v-invariant-classify}
Let $K$ be a $P_v$-invariant $\z$-lattice on $F$. Then there are unique $k_1, k_2 \in \Q_{>0}$ such that
\[
K=\begin{cases}
    \z (k_1v) + \z (k_2\sqrt{d}v) & \text{if } K \text{ is orthogonal w.r.t. } P_v,\\ 
    \z (k_1v) + \z (\frac{1}{2}k_1v+k_2\sqrt{d}v) & \text{otherwise}.
\end{cases}
\]
\end{prop}

\begin{proof}
First $k_1 v$ is uniquely determined as the basis element of a sublattice $K \cap \Q v \le \Q v$ directed towards the positive direction of $v$. Suppose that $k_1 v$ is extended to a basis $\{k_1 v, w\}$ of $K$ with $w$ directed towards $\sqrt{d} v$, meaning that $w = k_2 \sqrt{d} v + l v$ with $k_2>0$. Then $w$ is uniquely determined up to translation by a vector in $\z (k_1 v)$. Thus the choice of $w = k_2 \sqrt{d} v + l v$ satisfying $k_2>0$ and $0 \le l < k_1$ becomes unique.

Since $w + P_v w = 2l v \in K$, we have $2l v \in \Z(k_1 v) $. Hence $2l \in \Z k_1$. From $0 \le l < k_1$, we have either $l=0$ or $l = k_1/2$. If $l=0$, then we have $K = \z (k_1v) + \z (k_2\sqrt{d}v)$, and one may immediately see that this is orthogonal with respect to $P_v$. If $l = k_1/2$, then $K = \z (k_1v) + \z (\frac{1}{2}k_1v+k_2\sqrt{d}v)$. In this case, it follows that no element of the form $w'=k' \sqrt{d} v$ together with $k_1 v$ gives a basis of $K$. Thus $K$ is not orthogonal with respect to $P_v$.
\end{proof}
This characterization of $P_v$-invariant $\z$-lattices will be used in counting reflection invariant sublattices.

\begin{prop}\label{prop:latticesum.classify}
For a $\Z$-lattice $K$ on $F$ such that $P_v K = K$, we have
\[
\sum_{M \subseteq K, P_v M = M} |K:M|^{-s} = \begin{cases} (1+2^{-s}) \zeta(s)^2 &\text{if } K \text{ is orthogonal w.r.t }P_v, \\ (1 -  2^{-s} + 2 \cdot 2^{-2s}) \zeta(s)^2 & \text{otherwise.} \end{cases}
\]
\end{prop}

\begin{proof}\phantom{\qedhere}
We first consider the case when $K$ is orthogonal with respect to $P_v$. By Proposition \ref{prop:P_v-invariant-classify} we may write $K =\Z(k_1 v) + \Z (k_2 \sqrt{d} v)$ for some $k_1, k_2 \in \Q_{>0}$. Put $u_1 = k_1 v$ and $u_2 = k_2 \sqrt{d} v$. Let $M$ be any $P_v$-invariant $\z$-lattice on $F$. By applying Proposition \ref{prop:P_v-invariant-classify} to $M$, we may express $M$ as
\begin{equation}\label{eqn:M-basis}
    M=\begin{cases}
        \Z (l_1 u_1) + \Z (l_2 u_2) & \text{if $M$ is orthogonal with respect to $P_v$},\\
        \Z (2l_1 u_1) + \Z (l_1 u_1 + l_2 u_2) & \text{otherwise}
    \end{cases}
\end{equation}
where $l_1, l_2 \in \Q_{>0}$ are uniquely determined.

Among such $M$, we need to characterize all sublattices of $K$ and add $|K:M|^{-s}$ for all such $M$. If $M$ is orthogonal with respect to $P_v$, then $M = \Z l_1 u_1 + \Z l_2 u_2 \subseteq K$ if and only if $l_1 u_1, l_2 u_2\in K$, and hence $l_1, l_2 \in \Z$. Summing $|K:M|^{-s} = (l_1 l_2)^{-s}$ over all $l_1, l_2 \in \N$ gives $\zeta(s)^2$. If $M$ is not orthogonal with respect to $P_v$, then $M =\Z (2l_1 u_1) + \Z (l_1 u_1 + l_2 u_2) \subseteq K$ if and only if $2l_1, l_1, l_2 \in \Z$. Hence summing $|K:M|^{-s} = (2 l_1 l_2)^{-s}$ for all $l_1, l_2 \in \N$ gives $2^{-s} \zeta(s)^2$. Thus the whole summation of $|K:M|^{-s}$ over all $P_v$-invariant sublattices $M \subseteq K$ becomes $(1+2^{-s})\zeta(s)^2$.

Now we consider the case when $K$ is not orthogonal with respect to $P_v$. In this case we may write $K = \Z (2u_1) + \Z(u_1 + u_2)$ for a vector $u_1$ parallel to $v$ and a vector $u_2$ parallel to $\sqrt{d} v$. Note for $n_1,n_2\in\q$ that $n_1 u_1 + n_2 u_2 \in K$ if and only if $n_1, n_2 \in \Z$ and $n_1 \equiv n_2 \Mod{2}$.

If $M$ is orthogonal with respect to $P_v$, then $M = \Z (l_1 u_1) + \Z (l_2 u_2)$ by \eqref{eqn:M-basis}, and so $M\subseteq K$ if and only if $l_1, l_2 \equiv 0 \Mod{2}$. Thus summing $|K:M|^{-s} = (l_1 l_2/2)^{-s}$ over all such $l_1, l_2\in\n$ gives $2^{-s} \zeta(s)^2$. If $M$ is not orthogonal with respect to $P_v$, then $M = \Z (2l_1 u_1) + \Z (l_1 u_1 + l_2 u_2)$ by \eqref{eqn:M-basis}, and so $M \subseteq K$ is equivalent to that $l_1, l_2 \in \Z$ and $l_1 \equiv l_2 \Mod{2}$. Summing $|K:L|^{-s} = (l_1 l_2)^{-s}$ over all such $l_1, l_2\in\n$ gives $(1-2^{-s})^2\zeta(s)^2$ (when $l_1, l_2$ both odd) plus $2^{-2s}\zeta(s)^2$ (when $l_1, l_2$ both even). Thus the whole sum of $|K:M|^{-s}$ for $P_v$-invariant sublattices $M \subseteq K$ becomes
\[
2^{-s} \zeta(s)^2 + (1-2^{-s})^2 \zeta(s)^2 + 2^{-2s} \zeta(s)^2 = (1-2^{-s} + 2 \cdot 2^{-2s})\zeta(s)^2.\tag*{\qed}
\]
\end{proof}

Now we determine all orthogonal ideals in the following number-theoretic criterion:

\begin{prop}\label{prop:ortho.ideals}
Let $F=\Q(\sqrt{d})$ be an imaginary quadratic field with $d<0$ a square-free integer. Then orthogonal ideals can exist only when $d\equiv 2,3\Mod{4}$, i.e., $R = \Z[\sqrt{d}]$. In this case, a fractional ideal $I$ is orthogonal if and only if $[I]$ can be represented as a product of the ideal classes of ramified factors of the form $\Z r + \Z \sqrt{d}$, where $r$ is a prime dividing $d$.
\end{prop}

We remind that $R = \Z[\frac{1+\sqrt{d}}{2}]$ or $R=\Z[\sqrt{d}]$, depending on whether $d \equiv 1 \Mod{4}$ or not. Proposition \ref{prop:quadIdealBasis} allows us to determine all ramified factors of the form $\Z r + \Z \sqrt{d}$ as follows:
\begin{itemize}
\item $d \equiv 1 \Mod{4}$: none
\item $d \equiv 2 \Mod{4}$: all ramified factors
\item $d \equiv 3 \Mod{4}$: all ramified factors except the norm $2$ factor
\end{itemize}
In latter two cases, we will soon prove that classes of orthogonal ideals form a subgroup in $\Cl_F$.

\begin{proof}[Proof of Proposition \ref{prop:ortho.ideals}]
Suppose that $I$ is orthogonal with respect to $P_v$. Note that $v^{-1} I$ is orthogonal with respect to $P_1$. By multiplying a nonzero element to $I$, we may assume that $I = \Z 1 + \Z k\sqrt{d}$ for some $k \in \Q$. Since $I$ is a fractional ideal, we should have $\tau I \le I$, where $\tau$ is defined in \eqref{eqn:def-O_F} to satisfy $R = \Z + \Z\tau$. 

If $\tau = (1 +\sqrt{d})/2$, then $\tau \cdot 1 = \tau \notin I$, so $I$ cannot be an ideal. Hence $R=\z [\sqrt{d}]$ and $\tau = \sqrt{d}$. The condition $\tau I \le I$ is equivalent to that
\[
\sqrt{d} \in I \Leftrightarrow k \vert 1 \quad \text{and} \quad (k \sqrt{d}) \sqrt{d} = kd \in I \Leftrightarrow 1 \vert kd.
\]
Hence $k = \frac{1}{m}$ for some integer $m$ dividing $d$. As $d$ is square-free, we may write $m = \prod_{i \in S} r_i$ where $r_i$'s are distinct primes dividing $d$. Noting that 
\[
mI = \Z \prod_{i \in S} r_i + \Z \sqrt{d} = \prod_{i \in S}\left(\z r_i + \z \sqrt{d}\right),
\]
the ideal class $[I]$ is a product of ramified factors of the form $\z r_i+ \z\sqrt{d}$. 

Conversely, any product of the ideals $\z r_i + \z \sqrt{d}$ over primes $r_i \vert d$ with $i\in S$ has a $\Z$-basis $\left\{\prod_{i \in S} r_i, \sqrt{d}\right\}$. Hence it is orthogonal. This completes the proof.
\end{proof}

Let $\mathscr{I}_{\mathrm{refl}}$ be the set of all reflexive fractional ideals of $F$ and let us define the set
\[
[\mathscr{I}_{\mathrm{refl}}] := \left\{ [\mathfrak{a}] : \mathfrak{a} \in \mathscr{I}_{\mathrm{refl}}\right\} = \left\{ [\mathfrak{a}] \in \Cl_F : \ [\mathfrak{a}]=[\overline{\mathfrak{a}}]\right\}.
\]
Note that the last equality follows by Proposition \ref{prop:reflection-property} (3). Hence $[\mathscr{I}_{\mathrm{refl}}]$ is a subgroup of $\Cl_F$.
Moreover, we define the sets
\[
\mathscr{I}_{\mathrm{ram}} := \left\{ \prod_{\gamma_i \text{ ramified}} \gamma_i^{e_i} : e_i \in \{0,1\} \right\}
\]
and
\[
\mathscr{I}_{\mathrm{ortho}} = \begin{cases} \emptyset & \text{if } R = \Z\left[\frac{1+\sqrt{d}}{2}\right], \\
\left\{ \prod_{\gamma_i \text{:orthogonal,ramified} } \gamma_i^{e_i} : e_i \in \{0,1\} \right\} & \text{if } R = \Z[\sqrt{d}].
\end{cases}
\]
Also let us denote
\[
[\mathscr{I}_{\mathrm{ram}}] = \left\{ [\mathfrak{a}] : \mathfrak{a} \in \mathscr{I}_{\mathrm{ram}}\right\} \quad \text{and} \quad
[\mathscr{I}_{\mathrm{ortho}}] = \left\{ [\mathfrak{a}] : \mathfrak{a} \in \mathscr{I}_{\mathrm{ortho}}\right\}.
\]
One may observe that
\[
[\mathscr{I}_{\mathrm{ortho}}]\subseteq [\mathscr{I}_{\mathrm{ram}}] \subseteq [\mathscr{I}_{\mathrm{refl}}].
\]

\begin{lem}
Let $F=\Q(\sqrt{d})$ be an imaginary quadratic field, where $d<0$ is a square-free integer with $d \equiv 2,3 \Mod{4}$.
\begin{enumerate}[label={{\rm (\arabic*)}}, leftmargin=*]
    \item A fractional ideal $I$ is orthogonal if and only if $[I] \in [\mathscr{I}_{\mathrm{ortho}}]$.
    \item The set $[\mathscr{I}_{\mathrm{ram}}]$ is the subgroup of $\Cl_F$ generated by classes $[\gamma]$ of all ramified factors $\gamma$, and $[\mathscr{I}_{\mathrm{ortho}}]$ is the subgroup of $\Cl_F$ generated by classes of all ramified factors except the one above $2$.
    \item If $d \equiv 2 \Mod{4}$, then $[\mathscr{I}_{\mathrm{ram}}]=[\mathscr{I}_{\mathrm{ortho}}]$. If $d \equiv 3 \Mod{4}$, then $[\mathscr{I}_{\mathrm{ortho}}]$ is an index $2$ subgroup of $[\mathscr{I}_{\mathrm{ram}}]$, and $[\mathscr{I}_{\mathrm{ram}}]/[\mathscr{I}_{\mathrm{ortho}}]$ is generated by the ramified factor $\gamma_2$ of $R$ above $2$.
\end{enumerate}
\end{lem}

\begin{proof}
(1) This follows immediately from Proposition \ref{prop:ortho.ideals}.

\noindent (2) Noting that squares of ramified factors are always principal, one may easily check that both sets $\mathscr{I}_{\mathrm{ram}}$ and $\mathscr{I}_{\mathrm{ortho}}$ are closed under multiplication up to principal factor. Descriptions of their generators are also immediate from the constructions of $\mathscr{I}_{\mathrm{ram}}$ and $\mathscr{I}_{\mathrm{ortho}}$.

\noindent (3) If $d \equiv 2 \Mod{4}$, then all ramified factors are orthogonal. Hence $\mathscr{I}_{\mathrm{ram}}=\mathscr{I}_{\mathrm{ortho}}$ and thus $[\mathscr{I}_{\mathrm{ram}}]=[\mathscr{I}_{\mathrm{ortho}}]$. If $d \equiv 3 \Mod{4}$, then the ramified factor $\gamma_2 = \z2 + \z(\sqrt{d}-1)$ of $R$ above $2$ is nonorthogonal by Propositions \ref{prop:quadIdealBasis} and \ref{prop:ortho.ideals}. Hence $[\gamma_2] \in [\mathscr{I}_{\mathrm{ram}}] \setminus [\mathscr{I}_{\mathrm{ortho}}]$ by the part (1). Meanwhile $\mathscr{I}_{\mathrm{ram}}$ is a disjoint union of $\mathscr{I}_{\mathrm{ortho}}$ and $\gamma_2 \mathscr{I}_{\mathrm{ortho}}$. Hence we have $[\mathscr{I}_{\mathrm{ram}}] \subseteq [\mathscr{I}_{\mathrm{ortho}}] \cup [\gamma_2][\mathscr{I}_{\mathrm{ortho}}]$, and thus $|[\mathscr{I}_{\mathrm{ram}}]/[\mathscr{I}_{\mathrm{ortho}}]| \le 2$. Therefore, the index is exactly $2$ and the factor group is generated by $\gamma_2$.
\end{proof}

\subsection{Summing over all reflections - the case $|\of{\times}|=2$}

We will calculate the sum of reflection terms, under the setting of $L=I$ being a fractional ideal of $F=\q(\sqrt{d})$ with a square-free integer $d<0$.  We will distinguish the general case $|\of{\times}| =2$ and exceptional case $|\of{\times}| \neq 2$, as our strategy becomes rather different. We discuss the general case when $|\of{\times}|=2$, equivalently when $d \neq -1,3$.

\begin{rmk}
     To ease the notation, throughout the rest of this article the symbols $\prod_p$, $\prod_q$, and $\prod_r$ represent the products over all split primes $p$, all inert primes $q$, and all ramified primes $r$ of $F$, respectively, unless stated otherwise. 
\end{rmk}

\begin{thm} \label{thm:refTermGeneral}
Let $F=\q(\sqrt{d})$ be an imaginary quadratic field with $d \neq -1, -3$ being a negative square-free integer. Let $I$ be a fractional ideal of $F$.
\begin{enumerate}[label={{\rm (\arabic*)}}, leftmargin=*]

\item If $d \equiv 1 \Mod{4}$, then we have
\[
\mathrm{Refl}_I(s)= \frac{(1-2^{-s} + 2 \cdot 2^{-2s}) \zeta(s)^2}{2\zeta(2s) \prod_{r} (1+r^{-s})} \left( \sum_{n=1}^{\infty} \frac{ \# \{[J] : |I:J|=n, [J] = [\overline{J}] \} } {n^s} \right).
\]

\item If $d \equiv 2 \Mod{4}$, then we have
\begin{align*}
\mathrm{Refl}_I(s) &= \frac{(1-2^{-s} + 2 \cdot 2^{-2s}) \zeta(s)^2}{2\zeta(2s) \prod_{r} (1+r^{-s})} \left( \sum_{n=1}^{\infty} \frac{ \# \{[J] : |I:J|=n, [J] = [\overline{J}] \} } {n^s} \right) \\
&+ \frac{(2^{-s}-2^{-2s})\zeta(s)^2}{\zeta(2s) \prod_{r}(1+r^{-s})} \left( \sum_{n=1}^{\infty} \frac{ \# \{[J] \in [\mathscr{I}_{\mathrm{ram}}] : |I:J|=n \}}{n^{s}} \right),
\end{align*}
where the subgroup $[\mathscr{I}_{\mathrm{ram}}] \le \Cl_F$ is the group generated by classes of all ramified factors of $R$.

\item If $d \equiv 3 \Mod{4}$, then we have
\begin{align*}
\mathrm{Refl}_I(s) &= \frac{(1-2^{-s} + 2 \cdot 2^{-2s}) \zeta(s)^2}{2\zeta(2s) \prod_{r} (1+r^{-s})} \left( \sum_{n=1}^{\infty} \frac{ \# \{[J] : |I:J|=n, [J] = [\overline{J}] \} } {n^s} \right) \\
&+ \frac{2^{-s}\zeta(s)^2}{\zeta(2s) \prod_{r}(1+r^{-s})} \left( \sum_{n=1}^{\infty} \frac{ \# \{[J] \in [\mathscr{I}_{\mathrm{ortho}}] : |I:J|=n \}}{n^{s}} \right) \\
&-  \frac{2^{-2s}\zeta(s)^2}{\zeta(2s) \prod_{r}(1+r^{-s})} \left( \sum_{n=1}^{\infty} \frac{ \# \{[J] \in [\mathscr{I}_{\mathrm{ram}}] \setminus [\mathscr{I}_{\mathrm{ortho}}] : |I:J|=n \}}{n^{s}} \right),
\end{align*}
where the subgroup $[\mathscr{I}_{\mathrm{ram}}] \le \Cl_F$ is the group generated by classes of all ramified factors of $R$, and the subgroup $[\mathscr{I}_{\mathrm{ortho}}] \le \Cl_F$ is the group generated by classes of all ramified factors of $R$ except the one above $2$.
\end{enumerate}
\end{thm}

\begin{proof}
Substituting the formula \eqref{eqn:Z_I,rho} of $Z_{I,\rho}(s)$ into \eqref{eqn:def-Rot-Refl}, we have 
\begin{align}
\mathrm{Refl}_I(s) &= \sum_{\rho \in O^{-}(F)} Z_{I,\rho}(s)\nonumber \\
&=\sum_{\rho \in O^{-}(F)}\sum_{K \subseteq I, \rho(K) = K} \frac{|I : K|^{-s} } {\# \{\alpha \in O^{+}(F) : \alpha K \subseteq I \} + \# \{\alpha \in O^{+}(F) : \overline{ (\alpha K)} \subseteq I \} }\nonumber\\
&= \sum_{K \subseteq I } \frac{|I : K|^{-s}\# \{\rho \in O^{-}(F) : \rho(K) = K \} } {\# \{\alpha \in O^{+}(F) : \alpha K \subseteq I \} + \# \{\alpha \in O^{+}(F) : \overline{ (\alpha K)} \subseteq I \} } . \label{eqn:Refl-intermediate-0}
\end{align}

For a reflexive sublattice $K\subseteq I$, let $J=RK\le I$. Then $[J]=[\overline{J}]$ (i.e. $J$ is reflexive), $O^-(K)=O^-(J)$, and $\#\{\rho \in O^{-}(F) : \rho(K) = K \}=|O^-(K)|=2$ by Proposition \ref{prop:reflection-property}. Noting that $\alpha K \subseteq I \Leftrightarrow \alpha J\le I$ and $\overline{(\alpha K)}\subseteq I \Leftrightarrow \overline{(\alpha J)}\le I$ for any $\alpha\in O^+(F)$, we may further verify that 
\begin{align}
\mathrm{Refl}_I(s) & =\sum_{ \substack{K \subseteq I \\ K \text{:reflexive}}} \frac{2|I : K|^{-s} } {\# \{\alpha \in O^{+}(F) : \alpha K \subseteq I \} + \# \{\alpha \in O^{+}(F) : \overline{ (\alpha K)} \subseteq I \} } \nonumber \\
&=\sum_{J \le I, [J]=[\overline{J}]} \frac{2|I : J|^{-s} } {\# \{\alpha \in O^{+}(F) : \alpha J \le I \} + \# \{\alpha \in O^{+}(F) : \overline{ (\alpha J)} \le I \} } \Psi(J), \label{eqn:Refl-intermediate-1}
\end{align}
where $\Psi(J)$ is defined for a reflexive fractional ideal $J$ as
\begin{equation}\label{eqn:def-Psi}
\Psi(J):=  \sum_{ \substack{K \subseteq J, RK=J\\ O^{-}(K)=O^{-}(J)}} |J:K|^{-s}.
\end{equation}
Note that $\Psi(J)$ depends only on the ideal class of $J$. Hence for an ideal class $g \in [\mathscr{I}_{\mathrm{refl}}]$ with $g=[J]$ and $J\in\mathscr{I}_{\mathrm{refl}}$, we may define $\Psi(g) = \Psi(J)$.

Moreover, since $[J]=[\overline{J}]$ and $|\of{\times}|=2$, plugging \eqref{eqn:counting-oftimes-to-one-1} and \eqref{eqn:counting-oftimes-to-one-2} into \eqref{eqn:Refl-intermediate-1} implies that 
\[
\mathrm{Refl}_I(s)  = \sum_{J \le I, [J]=[\overline{J}]} \frac{ \Psi([J])|I:J|^{-s}}{2 \# \{ J_1 \le I : |I:J_1| = |I : J|, [J]=[J_1] \} }.
\]
Letting $n = |I:J|$, we obtain
\begin{align*}
\mathrm{Refl}_I(s)   &= \frac{1}{2} \sum_{n=1}^{\infty} \frac{1}{n^{s}}
\sum_{J \le I, [J]=[\overline{J}], |I:J|=n} \frac{\Psi([J])}{\# \{J_1 \le I, |I:J_1|=n, [J]=[J_1] \} } \\
&=\frac{1}{2} \sum_{n=1}^{\infty} \frac{1}{n^{s}} \sum_{ g\in \left\{[J]: J \le I, [J] = [\overline{J}], |I:J|=n\right\}} \Psi(g).
\end{align*}
We substitute Proposition \ref{prop:E[I]} to the above to have 
\begin{align*}
\mathrm{Refl}_I(s) 
&=  \frac{1}{2} \sum_{n=1}^{\infty} \frac{1}{n^{s}} \sum_{ g\in \left\{[J]: J \le I, [J] = [\overline{J}], |I:J|=n\right\} } \left( \frac{(1-2^{-s} + 2 \cdot 2^{-2s} + E(g))\zeta(s)^2}{\zeta(2s) \prod_{r \text{ ramified}} (1+r^{-s})} \right) \\
& = \frac{(1-2^{-s} + 2 \cdot 2^{-2s} )\zeta(s)^2}{2\zeta(2s) \prod_{r } (1+r^{-s})} \sum_{n=1}^{\infty} \frac{1}{n^{s}} \sum_{ g\in \left\{[J]: J \le I, [J] = [\overline{J}], |I:J|=n\right\} } 1 \\
&+ \frac{\zeta(s)^2}{2 \zeta(2s) \prod_{r} (1+r^{-s})} \sum_{n=1}^{\infty} \frac{1}{n^{s}} \sum_{ g\in \left\{[J]: J \le I, [J] = [\overline{J}], |I:J|=n\right\} } E(g)
\end{align*}
where $E(g)$ is defined in Proposition \ref{prop:E[I]}. Substituting appropriate values of $E(g)$ for different cases yields the theorem.
\end{proof}

\subsection{Enumerating the series $\Psi$}
Analogous to the proof of Lemma \ref{lem:sumideals}, we would like to find some suitable linear equations in $\Psi(g)$ defined in \eqref{eqn:def-Psi} for $g\in [\mathscr{I}_{\mathrm{refl}}]$ to verify them.

Let $g\in[\mathscr{I}_{\mathrm{refl}}]$ and let $I$ be a reflexive ideal such that $g=[I]$. Suppose that $P_v \in O^{-}(I)$, that is, $P_v I=I$. Then since $|R^{\times}|=2$, Proposition \ref{prop:reflection-property} (5) implies that $O^{-}(I) = \{P_v, P_{\sqrt{d} v}\}$ and that for a $\z$-lattice $K$ we have $P_v K = K$ if and only if $O^{-}(K) = O^{-}(I)$. Moreover, since $I$ is orthogonal with respect to $P_v$ if and only if it is orthogonal with respect to $P_{\sqrt{d} v}$, we have that $I$ is orthogonal if and only if it is orthogonal with respect to $P_v$. Hence applying Proposition \ref{prop:latticesum.classify} to $I$, we have
\begin{equation}\label{eqn:sum-of-indices-reflexive-1}
\sum_{K \subseteq I, O^{-}(K)=O^{-}(I)} |I:K|^{-s} = \begin{cases} (1+2^{-s})\zeta(s)^2 & \text{if } I \text{ is orthogonal}, \\ (1-2^{-s} + 2 \cdot 2^{-2s}) \zeta(s)^2 & \text{otherwise.} \end{cases}
\end{equation}
Letting $J=RK$ and noting that $O^{-}(J)=O^{-}(K)$ by Proposition \ref{prop:reflection-property} (5), we have
\begin{align}
\sum_{K \subseteq I, O^{-}(K) = O^{-}(I)} |I:K|^{-s} &= \sum_{J \le I, O^{-}(J) = O^{-}(I)} |I:J|^{-s} \sum_{K \subseteq J, RK=J, O^{-}(K)= O^{-}(J)} |J:K|^{-s} \nonumber \\ 
&= \sum_{J \le I, O^{-}(J) = O^{-}(I) } |I:J|^{-s}\Psi([J]). \label{eqn:sum-of-indices-reflexive-2}
\end{align}

Meanwhile, Proposition \ref{prop:reflexive-TFAE} implies that $O^{-}(J) = O^{-}(I)$ if and only if $J = IM$ for some ideal $M = \overline{M}$. Also note that $J \le I$ if and only if $M \le R$, so one may follow the proof of Proposition \ref{prop:reflexive-TFAE} to show that a proper ideal $M$ satisfying $M=\overline{M}$ can be uniquely represented as 
\[
M = m \mathfrak{a} \quad \text{with } m \in \N, \mathfrak{a} \in \mathscr{I}_{\mathrm{ram}}.
\]
We substitute this description of $M$ to $J=IM$ and note that 
\[
|I:J| = N(M) = N(m \mathfrak{a}) = m^2 N(\mathfrak{a}).
\]
Plugging these into \eqref{eqn:sum-of-indices-reflexive-2}, we have 
\begin{equation}\label{eqn:sum-of-indices-reflexive-3}
\sum_{K \subseteq I, O^{-}(K) = O^{-}(I)} |I:K|^{-s} = \sum_{m \in \N, \mathfrak{a} \in \mathscr{I}_{\mathrm{ram}}} m^{-2s} N(\mathfrak{a})^{-s} \Psi(g [\mathfrak{a}]) = \zeta(2s) \sum_{\mathfrak{a} \in \mathscr{I}_{\mathrm{ram}}} N(\mathfrak{a})^{-s} \Psi(g [\mathfrak{a}]).
\end{equation}
Combining \eqref{eqn:sum-of-indices-reflexive-1} and \eqref{eqn:sum-of-indices-reflexive-3}, we obtain the following equation in $\Psi$ for each where $g \in [\mathscr{I}_{\mathrm{refl}}]$:
\begin{equation}\label{eqn:PsiEqn}
\sum_{\mathfrak{a} \in \mathscr{I}_{\mathrm{ram}}} N(\mathfrak{a})^{-s} \Psi(g[\mathfrak{a}]) = \frac{\zeta(s)^2}{\zeta(2s)} \cdot \begin{cases} 1+2^{-s} & \text{if } g \in [\mathscr{I}_{\mathrm{ortho}}], \\ (1-2^{-s} + 2 \cdot 2^{-2s}) & \text{otherwise}. \end{cases}
\end{equation}
We show that \eqref{eqn:PsiEqn} can completely determine $\Psi$ as follows.

\begin{prop}\label{prop:E[I]}
Let $g\in [\mathscr{I}_{\mathrm{refl}}]$. Then we have
\[
\Psi(g) = \frac{(1-2^{-s} + 2 \cdot 2^{-2s} + E(g))\zeta(s)^2}{\zeta(2s) \prod_{r } (1+r^{-s})},
\]
where
\[
E(g) = \begin{cases} 0 & \text{if } d \equiv 1 \Mod{4} \text{ or } g \notin [\mathscr{I}_{\mathrm{ram}}], \\
2 \cdot 2^{-s} - 2 \cdot 2^{-2s} & \text{if } d \equiv 2 \Mod{4} \text{ and } g \in [\mathscr{I}_{\mathrm{ram}}], \\
2\cdot 2^{-s} & \text{if } d \equiv 3 \Mod{4} \text{ and } g \in [\mathscr{I}_{\mathrm{ortho}}], \\
-2 \cdot 2 ^{-2s} & \text{if } d \equiv 3 \Mod{4} \text{ and } g \in [\mathscr{I}_{\mathrm{ram}}] \setminus [\mathscr{I}_{\mathrm{ortho}}].
\end{cases}
\]
\end{prop}

\begin{proof}
We first show that $\Psi$ satisfying \eqref{eqn:PsiEqn} is determined uniquely. Suppose $\Psi_1(g)$ and $\Psi_2(g)$ are two families of Dirichlet series in $s$ indexed by $\Cl_F$ which both satisfy \eqref{eqn:PsiEqn}. Denoting $\mathfrak{X}(g) = \Psi_1(g) - \Psi_2(g)$ as their differences, we have
\[
\sum_{\mathfrak{a} \in \mathscr{I}_{\mathrm{ram}}} N(\mathfrak{a})^{-s} \mathfrak{X}(g[\mathfrak{a}])=0
\]
for any $g \in [\mathscr{I}_{\mathrm{refl}}]$. If not all $\mathfrak{X}(g)$ are identically zero, one can pick the smallest $m \in \N$ such that there exists an $g_0 \in [\mathscr{I}_{\mathrm{refl}}]$ such that the coefficient of $m^{-s}$ in $\mathfrak{X}(g_0)$ is nonzero. Then in the equation
\[
\sum_{\mathfrak{a} \in \mathscr{I}_{\mathrm{ram}}} N(\mathfrak{a})^{-s} \mathfrak{X}(g_0 [\mathfrak{a}])= \mathfrak{X}(g_0) + \sum_{\mathfrak{a} \in \mathscr{I}_{\mathrm{ram}} \setminus (1)} N(\mathfrak{a})^{-s} \mathfrak{X}(g_0[\mathfrak{a}]) =0,
\]
the coefficient of $m^{-s}$ in $N(\mathfrak{a})^{-s} \mathfrak{X}(g_0[\mathfrak{a}])$ is zero for any $\mathfrak{a} \neq (1)$ as $N(\mathfrak{a}) >1$, while coefficient of $m^{-s}$ in $\mathfrak{X}(g_0)$ is nonzero. Hence this equation cannot hold, obtaining a contradiction. It follows that $\mathfrak{X}=0$ identically and a family of Dirichlet series $\Psi(g)$ satisfying \eqref{eqn:PsiEqn} is unique if it exists.

Thus it suffices to show that the suggested value of $\Psi$ in the statement of the proposition satisfies all equations of \eqref{eqn:PsiEqn}. Denote the value of $\Psi$ speculated in the statement of this proposition as $\Psi_1$. For each cases, we will substitute $\Psi_1$ to the left hand side of \eqref{eqn:PsiEqn}, and calculate that it becomes identical to the right hand side of \eqref{eqn:PsiEqn} in that case. We first note that
\begin{equation}\label{eqn:sum-norm-ram}
\sum_{\mathfrak{a} \in \mathscr{I}_{\mathrm{ram}} } N(\mathfrak{a})^{-s} = \prod_{r } (1+r^{-s}).
\end{equation}
Recall that the symbol $\prod_r$ indicates the product over all ramified primes $r$ of $R$, unless specified otherwise.

{\bf Case 1) $d \equiv 1 \Mod{4}$.\rm} In this case, we always have
\[
\Psi_1(g[\mathfrak{a}] ) =  \frac{(1-2^{-s} + 2 \cdot 2^{-2s}) \zeta(s)^2}{\zeta(2s) \prod_{r}(1+r^{-s})}.
\]
Hence substituting this to the left hand side of \eqref{eqn:PsiEqn} and using \eqref{eqn:sum-norm-ram}, we have
\begin{align*}
\sum_{\mathfrak{a} \in \mathscr{I}_{\mathrm{ram}}} N(\mathfrak{a})^{-s} \Psi_1(g[\mathfrak{a}])  &= \frac{(1-2^{-s} + 2 \cdot 2^{-2s}) \zeta(s)^2}{\zeta(2s) \prod_{r}(1+r^{-s})} \left( \sum_{\mathfrak{a} \in \mathscr{I}_{\mathrm{ram}}} N(\mathfrak{a})^{-s} \right) \\
&= \frac{(1-2^{-s} + 2 \cdot 2^{-2s}) \zeta(s)^2}{\zeta(2s) \prod_{r}(1+r^{-s})} \prod_{r}(1+r^{-s}) \\
&= \frac{(1-2^{-s} + 2 \cdot 2^{-2s}) \zeta(s)^2}{\zeta(2s)}.
\end{align*}
This coincides with the right hand side of \eqref{eqn:PsiEqn} as $[\mathscr{I}_{\mathrm{ortho}}] = \emptyset$.

{\bf Case 2) $d \equiv 2,3 \Mod{4}$ and $g \notin [\mathscr{I}_{\mathrm{ram}}]$.} In this case, we note that $g[\mathfrak{a}] \notin [\mathscr{I}_{\mathrm{ram}}]$ for any $\mathfrak{a} \in \mathscr{I}_{\mathrm{ram}}$ since $[\mathscr{I}_{\mathrm{ram}}]$ is subgroup. Hence $g[\mathfrak{a}] \notin [\mathscr{I}_{\mathrm{ortho}}]$. Thus we have 
\[
\Psi_1(g[\mathfrak{a}] ) =  \frac{(1-2^{-s} + 2 \cdot 2^{-2s}) \zeta(s)^2}{\zeta(2s) \prod_{r}(1+r^{-s})}
\]
for all $\Psi_1(g[\mathfrak{a}])$ in the left hand side of \eqref{eqn:PsiEqn}. Thus \eqref{eqn:PsiEqn} hold for this case as verified in {\bf Case 1\rm}.

{\bf Case 3) $d \equiv 2 \Mod{4}$ and $g \in [\mathscr{I}_{\mathrm{ram}}] = [\mathscr{I}_{\mathrm{ortho}}]$.\rm}
In this case, we have $g[\mathfrak{a}] \in [\mathscr{I}_{\mathrm{ram}}]$ for all $\mathfrak{a} \in \mathscr{I}_{\mathrm{ram}}$ since as $[\mathscr{I}_{\mathrm{ram}}]$ is subgroup. Hence we have
\[
\Psi_1(g[\mathfrak{a}] ) =  \frac{(1+ 2^{-s}) \zeta(s)^2}{\zeta(2s) \prod_{r}(1+r^{-s})}.
\]
Thus plugging this into the left hand side of \eqref{eqn:PsiEqn} and using \eqref{eqn:sum-norm-ram}, we have
\begin{align*}
\sum_{\mathfrak{a} \in \mathscr{I}_{\mathrm{ram}}} N(\mathfrak{a})^{-s} \Psi_1(g[\mathfrak{a}])  &= \frac{(1+2^{-s} ) \zeta(s)^2}{\zeta(2s) \prod_{r}(1+r^{-s})} \left( \sum_{\mathfrak{a} \in \mathscr{I}_{\mathrm{ram}}} N(\mathfrak{a})^{-s} \right) \\
&= \frac{(1+2^{-s}) \zeta(s)^2}{\zeta(2s) \prod_{r}(1+r^{-s})} \prod_{r}(1+r^{-s}) \\
&= \frac{(1+2^{-s}) \zeta(s)^2}{\zeta(2s)}.
\end{align*}
This coincides with the right hand side of \eqref{eqn:PsiEqn} as $g \in [\mathscr{I}_{\mathrm{ortho}}] = [\mathscr{I}_{\mathrm{ram}}]$.

{\bf Case 4) $d \equiv 3 \Mod{4}$ and $g \in  [\mathscr{I}_{\mathrm{ortho}}]$.\rm} Denote by $\gamma_2$ be the ramified factor over $2$. Recall that the set $\mathscr{I}_{\mathrm{ram}}$ is partitioned into $\mathscr{I}_{\mathrm{ortho}}$ and $\gamma_2 \mathscr{I}_{\mathrm{ortho}}$, and $[\gamma_2]$ becomes the nontrivial class in $[\mathscr{I}_{\mathrm{ram}}]/[\mathscr{I}_{\mathrm{ortho}}]$. Thus we have $g[\mathfrak{a}] \in [\mathscr{I}_{\mathrm{ortho}}] \Leftrightarrow \mathfrak{a} \in \mathscr{I}_{\mathrm{ortho}}$. Hence we divide the left hand side sum of \eqref{eqn:PsiEqn} depending on whether $\mathfrak{a} \in \mathscr{I}_{\mathrm{ortho}}$ and $\mathfrak{a} \in \gamma_2 \mathscr{I}_{\mathrm{ortho}}$. The sum over $\mathfrak{a} \in \mathscr{I}_{\mathrm{ortho}}$ is given as
\begin{align*}
\sum_{\mathfrak{a} \in \mathscr{I}_{\mathrm{ortho}}} N(\mathfrak{a})^{-s} \Psi_1(g[\mathfrak{a}])   &= \frac{(1+2^{-s} + 2 \cdot 2^{-2s}) \zeta(s)^2}{\zeta(2s) \prod_{r}(1+r^{-s})} \left( \sum_{\mathfrak{a} \in \mathscr{I}_{\mathrm{ortho}}} N(\mathfrak{a})^{-s} \right) \\
&= \frac{(1+2^{-s} + 2 \cdot 2^{-2s}) \zeta(s)^2}{\zeta(2s) \prod_{r}(1+r^{-s})} \prod_{r \neq 2}(1+r^{-s}) \\
&= \frac{(1+2^{-s} + 2 \cdot 2^{-2s}) \zeta(s)^2}{(1+2^{-s})\zeta(2s)}
\end{align*}
and the sum over $\mathfrak{a} \in \gamma_2 \mathscr{I}_{\mathrm{ortho}}$ is given as
\begin{align*}
\sum_{\mathfrak{a} \in \gamma_2 \mathscr{I}_{\mathrm{ortho}}} N(\mathfrak{a})^{-s} \Psi_1(g[\mathfrak{a}])    &= \frac{(1-2^{-s}) \zeta(s)^2}{\zeta(2s) \prod_{r}(1+r^{-s})} \left( \sum_{\mathfrak{a} \in \gamma_2 \mathscr{I}_{\mathrm{ortho}}} N(\mathfrak{a})^{-s} \right) \\
&= \frac{(1-2^{-s}) \zeta(s)^2}{\zeta(2s) \prod_{r}(1+r^{-s})} 2^{-s}\prod_{r \neq 2}(1+r^{-s}) \\
&= \frac{2^{-s}(1-2^{-s}) \zeta(s)^2}{(1+2^{-s})\zeta(2s)}.
\end{align*}
Adding those two gives
\[
\frac{(1+2^{-s} + 2 \cdot 2^{-2s}) \zeta(s)^2}{(1+2^{-s})\zeta(2s)}+\frac{2^{-s}(1-2^{-s}) \zeta(s)^2}{(1+2^{-s})\zeta(2s)} = \frac{(1+2 \cdot 2^{-s} + 2^{-2s}) \zeta(s)^2}{(1+2^{-s})\zeta(2s)} = \frac{(1+2^{-s})\zeta(s)^2}{\zeta(2s)}.
\]
As $g \in [\mathscr{I}_{\mathrm{ortho}}]$, this is identical to the desired right hand side of \eqref{eqn:PsiEqn}.

{\bf Case 5) $d \equiv 3 \Mod{4}$ and $g \in [\mathscr{I}_{\mathrm{ram}}] \setminus [\mathscr{I}_{\mathrm{ortho}}]$.\rm} The proof for this case is similar to that of {\bf Case 4\rm}, but only the value of $\Psi_1$ changes. We have
\[
\sum_{\mathfrak{a} \in \mathscr{I}_{\mathrm{ortho}}} N(\mathfrak{a})^{-s} \Psi_1(g[\mathfrak{a}]) = \frac{(1-2^{-s}) \zeta(s)^2}{(1+2^{-s})\zeta(2s)}
\]
and
\[
\sum_{\mathfrak{a} \in \gamma_2 \mathscr{I}_{\mathrm{ortho}}} N(\mathfrak{a})^{-s} \Psi_1(g[\mathfrak{a}]) =\frac{2^{-s}(1+2^{-s} + 2 \cdot 2^{-2s}) \zeta(s)^2}{(1+2^{-s})\zeta(2s)}.
\]
Adding those two gives
\[
\frac{(1-2^{-s}) \zeta(s)^2}{(1+2^{-s})\zeta(2s)} + \frac{2^{-s}(1+2^{-s} + 2 \cdot 2^{-2s}) \zeta(s)^2}{(1+2^{-s})\zeta(2s)} = \frac{(1+2^{-2s} + 2 \cdot 2^{-3s}) \zeta(s)^2}{(1+2^{-s})\zeta(2s)} = \frac{(1-2^{-s} + 2 \cdot 2^{-2s}) \zeta(s)^2}{\zeta(2s)}.
\]
As $g \notin [\mathscr{I}_{\mathrm{ortho}}]$, this is identical to the desired right hand side of \eqref{eqn:PsiEqn}.
\end{proof}

\subsection{Exceptional cases: $d=-1$ and $d=-3$}

We evaluate $\mathrm{Refl}_I(s)$ for two exceptional cases when $\of{\times} \neq \{1, -1\}$, that is, when $R = \Z[\sqrt{-1}]$ and $R = \Z[\omega]$ with $\omega=\frac{-1+\sqrt{-3}}{2}$.

Note that in both cases, any fractional ideals $I$ is reflexive by Proposition \ref{prop:reflection-property} (3) since $[I]=[R]=[\overline{I}]$ as $R$ is PID. Hence $I$ satisfies $|O^{-}(I)| = |O^{-}(R)|$. The main difference from the general case is that elements of $O^{-}(K)$ of a sublattice $K \subseteq I$ can be of different sizes now.

\begin{prop} \label{prop:reflection_exceptional}
Let $R$ be either $\Z[i]$ or $\Z[\omega]$. Let $K$ be a reflexive $\z$-lattice on $R$ with $RK =R$.
\begin{enumerate}[label={\rm(\arabic*)}, leftmargin=*]
\item The reflection symmetry $O^{-}(R)$ is given explicitly as
\[
O^{-}\left(\Z[i]\right) = \left\{P_1, P_{i}, P_{1+i}, P_{-1+i} \right\} \quad \text{and} \quad O^{-}\left(\Z[\omega]\right) = \left\{P_1, P_{\sqrt{-3}}, P_{\omega}, P_{\sqrt{-3} \omega}, P_{\omega^2}, P_{\sqrt{-3}\omega^2} \right\}.
\]
\item We have $O^-(K)\subseteq O^-(R)$; and either $|O^-(K)|=2$ or $O^-(K)=O^-(R)$. The latter case $O^{-}(K) = O^{-}(R)$ occurs if and only if $K = R$.
\end{enumerate}
\end{prop}

\begin{proof}
(1) Noting that $P_1 \in O^{-}(R)$ and $O^{-}(R)$ is an $O^{+}(R)$-coset by Proposition \ref{prop:reflection-property}. Thus $O^-(R)=P_1 O^{+}(R)$, hence the proposition follows by Proposition \ref{prop:oF.characterization} (3).

(2) By Proposition \ref{prop:reflection-property} (1) and (2), we have $O^{-}(K) \subseteq O^{-}(RK) = O^{-}(R)$ and $O^{-}(K)$ has at least two elements.
Assume that $|O^-(K)|>2$. Note that if $v$ and $w$ form an angle $\theta$ as two vectors on the complex plane, then the composition of $P_v$ and $P_w$ is the rotation by the angle $2\theta$. Thus $O^+(K)$ has a rotation by an angle other than $0$ or $\pi$. This implies that $O^+(K)=\left\{u_\alpha : \alpha \in \of{\times}\right\}$ (see Proposition \ref{prop:oF.characterization} (3)). Noting that the elements of $\of{\times}$ span $R$ additively, we have $RK=K$. Thus we have $K=RK=R$.
\end{proof}

\begin{prop} \label{prop:refTermExceptional} Let $F$ be either $\q(\sqrt{-1})$ or $\q(\sqrt{-3})$, and let $R=\of{}$. We have
\[
\mathrm{Refl}_R(s) = \begin{cases} \frac{1}{2} (1-2^{-s} + 2 \cdot 2^{-2s}) \zeta(s)^2 \prod_{p \equiv 1 \Mod{3}} (1+p^{-s}) & \text{if } R = \Z[\omega], \\ \frac{1}{2} (1+2^{-2s}) \zeta(s)^2 \prod_{p \equiv 1 \Mod{4}} (1+p^{-s}) & \text{if } R = \Z[i].
\end{cases}
\]
\end{prop}
\begin{proof}
Since $R$ is PID, it suffice to consider the case when $I=R$ in calculating $\mathrm{Refl}_I(s)$. As in \eqref{eqn:Refl-intermediate-0}, we have
\[
\mathrm{Refl}_R(s) = \sum_{K \subseteq R } \frac{|R : K|^{-s} \# \{\rho \in O^{-}(F) : \rho (K) = K \} } {\# \{\alpha \in O^{+}(F) : \alpha K \subseteq R \} + \# \{\alpha \in O^{+}(F) : \overline{ (\alpha K)} \subseteq R \} } .
\]
Considering an ideal $I=RK\le R$, we have 
\begin{equation}\label{eqn:Refl-exceptional-intermediate-1}
\mathrm{Refl}_R(s) =\sum_{I \le R} \sum_{K \subseteq I, RK=I}  \frac{|R:I|^{-s} |I : K|^{-s} |O^{-}(K)| } {\# \{\alpha \in O^{+}(F) : \alpha K \subseteq R \} + \# \{\alpha \in O^{+}(F) : \overline{ (\alpha K)} \subseteq R \} }.
\end{equation}
Noting that $\alpha K \subseteq R \Leftrightarrow \alpha I\le R$ and $\overline{(\alpha K)}\subseteq R \Leftrightarrow \overline{(\alpha I)}\le R$ for any $\alpha\in O^+(F)$, and that $R$ is a PID, the equations \eqref{eqn:counting-oftimes-to-one-1} and \eqref{eqn:counting-oftimes-to-one-2} (with substituting $R, RK, J$ into $I, J, J_1$ respectively) imply that the denominator of the summand in \eqref{eqn:Refl-exceptional-intermediate-1} becomes 
\begin{equation}\label{eqn:Refl-exceptional-denominator}
\# \{\alpha \in O^{+}(F) : \alpha K \subseteq R \} + \# \{\alpha \in O^{+}(F) : \overline{ (\alpha K)} \subseteq R \}=2 |\of{\times}| \#\{J \le R  : N(J)=N(I) \}.
\end{equation}

On the other hand, as $R$ is a PID, we may write any $I\le R$ as $I=(t)$ for some $t\in R$. Then for a submodule $K\subseteq I$ with $RK=I$, we have
\[
    t^{-1}K\subseteq R, \quad R(t^{-1} K) =  t^{-1}RK=t^{-1}I=R \quad \text{and} \quad O^-(t^{-1}K) = t^{-1}O^-(K)t.
\]
Hence plugging \eqref{eqn:Refl-exceptional-denominator} into \eqref{eqn:Refl-exceptional-intermediate-1} first, and then replacing $K$ with $K'=t^{-1}K$, we have
\begin{align}
\mathrm{Refl}_R(s)=& \sum_{I \le R} \sum_{K \subseteq I, RK=I}  \frac{|R:I|^{-s} |I : K|^{-s} |O^{-}(K)|} {2 |\of{\times}| \# \{J \le R  : N(I)=N(J) \} } \nonumber \\
=& \frac{1}{2 |\of{\times}|} \sum_{I \le R} \sum_{K' \subseteq R, RK' = R} \frac{N(I)^{-s} |R : K'|^{-s} |O^{-}(K')| }{\# \{J \le R  : N(I)=N(J) \}} \nonumber \\
=& \frac{1}{2|\of{\times}|} \left( \sum_{I \le R}\frac{N(I)^{-s}}{ \# \{J \le R  : N(I)=N(J) \}} \right) \Omega_R, \label{eqn:Refl-exceptional-intermediate-2}
\end{align}
where
\[
\Omega_R :=   \sum_{K' \subseteq R, RK' = R}  |R:K'|^{-s} |O^{-}(K')|.
\]
Using \eqref{eqn:sumover-mathcal(N)(R)}, we may verify that
\begin{align}\label{eqn:Refl-innersum-norms}
\sum_{I \le R}\frac{N(I)^{-s}}{ \# \{J \le R  : N(I)=N(J) \}} = \sum_{n\in \mathcal{N}(R)} n^{-s}=  \zeta(2s) \prod_{p}(1+p^{-s}) \prod_{r} (1+r^{-s}). 
\end{align}

To calculate $\Omega_R$, let us consider the following sum and compute it in two different ways:
\[
S_R := \sum_{K \subseteq R} |R : K|^{-s} |O^{-}(K)|.
\]
First since $O^-(K)\subseteq O^-(R)$ by Proposition \ref{prop:reflection_exceptional} (2), we may write it as
\[
S_R = \sum_{\rho \in O^{-}(R)} \sum_{K \subseteq R, \rho (K) = K} |R:K|^{-s}.
\]
Note that we know the explicit description of $O^{-}(R)$ by Proposition \ref{prop:reflection_exceptional}, and that the inner sum for each $\rho \in O^{-}(R)$ may be obtained by Proposition \ref{prop:latticesum.classify}. In the case when $R = \Z[\omega]$, one may observe that $R$ is nonorthogonal. Hence
\[
S_{\Z[\omega]} = |\of{\times}| \left(1- 2^{-s} + 2 \cdot 2^{-2s}\right) \zeta(s)^2.
\]
When $R = \Z[i]$, one may observe that $R$ is orthogonal with respect to $P_1$ and $P_{i}$; and it is not orthogonal with respect to $P_{1+i}$ and $P_{-1+i}$. Hence
\[
S_{\Z[i]} = 4\left(1 + 2^{-2s}\right) \zeta(s)^2=|\of{\times}| \left(1 + 2^{-2s}\right) \zeta(s)^2.
\]
Therefore we have
\begin{equation}\label{eqn:S_R-formula-1}
S_R = |\of{\times}| f_R \zeta(s)^2, \quad  \text{where } f_R = \begin{cases} 1-2^{-s} + 2 \cdot 2^{-2s} & \text{if } R = \Z[\omega], \\ 1+ 2^{-2s} & \text{if } R = \Z[i]. \end{cases}
\end{equation}

To compute $S_R$ in another way, let us consider $I=RK$. Note that if $K$ is reflexive then $O^{-}(I) \cap O^{-}(R) \neq \emptyset$. Hence by Proposition \ref{prop:reflexive-TFAE}, $I=IR^{-1}$ is a product of powers of ramified factors and a rational number. Noting that there is a unique ramified ideal $\gamma$ of $R$ in both cases, the ideal $I$ should be of the form $I = n \gamma$ or $I = nR$ for some $n \in \Z_{>0}$. Thus
\begin{align}
S_R &= \sum_{K \subseteq R, \, K\text{:reflexive}} |R : K|^{-s}\# \{\rho \in O^{-}(R) : \rho (K) = K \} \nonumber \\
&= \sum_{I = nR \text{ or } n \gamma} N(I)^{-s} \sum_{K \subseteq R, RK = I} |I:K|^{-s} \{\rho \in O^{-}(R)  : \rho (K) = K \} \nonumber \\
&= \sum_{I=nR \text{ or } n \gamma} N(I)^{-s} \Omega_R \nonumber \\
&= \zeta(2s) (1+ N(\gamma)^{-s}) \Omega_R.\label{eqn:S_R-formula-2}
\end{align}
Combining \eqref{eqn:S_R-formula-1} and \eqref{eqn:S_R-formula-2}, and writing $1+N(\gamma)^{-s}$ as $\prod_r (1+r^{-s})$, we obtain 
\begin{equation}\label{eqn:Omega_R-formula}
\Omega_R  =\frac{|\of{\times}| f_R \zeta(s)^2}{\zeta(2s) \prod_r (1+r^{-s}) }, \quad  \text{where } f_R = \begin{cases} 1-2^{-s} + 2 \cdot 2^{-2s} & \text{if } R = \Z[\omega], \\ 1+ 2^{-2s} & \text{if } R = \Z[i]. \end{cases}
\end{equation}
Plugging \eqref{eqn:Refl-innersum-norms} and \eqref{eqn:Omega_R-formula} back into \eqref{eqn:Refl-exceptional-intermediate-2}, we have
\[
\mathrm{Refl}_R(s)=\frac{1}{2|\of{\times}|} \cdot \zeta(2s) \prod_{p}(1+p^{-s}) \prod_{r} (1+r^{-s}) \cdot \frac{|\of{\times}| f_R \zeta(s)^2}{\zeta(2s) \prod_r (1+r^{-s}) }=\frac{1}{2} f_R  \zeta(s)^2 \prod_{p}(1+p^{-s}).
\]
This completes the proof of the proposition if noting the definition of $f_R$ in \eqref{eqn:Omega_R-formula} and that the split primes of $R= \Z[\omega]$ and $\R = \Z[i]$ are given as $p \equiv 1 \Mod{3}$ and $p \equiv 1 \Mod{4}$, respectively.
\end{proof}

\section{Explicit computations of zeta functions of various binary lattices}\label{sec:examples}
In this section, we give explicit formulas for $\slzeta{L}{s}$ and $\glzeta{L}{s}$ for several binary $\z$-lattices $L$. As the formulas depend heavily on the imaginary quadratic field $F=\q(\sqrt{-4d_L})$, we provide examples according to the discriminant of corresponding quadratic fields $F$ (specifically, when $D=-3,-4,-7,-8,-20,-23$). For each discriminant, we provide the explicit formulas of $\slzeta{L}{s}$ and $\glzeta{L}{s}$ for the binary $\z$-lattices $L$ corresponding to ideal classes $[I]\in \Cl_F$ via Proposition \ref{prop:binary-correspond}. Recall that  the symbols $\prod_p$, $\prod_q$, and $\prod_r$ represent the products over all split primes $p$, all inert primes $q$, and all ramified primes $r$ of $F$, respectively. 

Before discussing individual examples, we would like to discuss some results for terms appearing in descriptions of $\slzeta{L}{s}$ and $\glzeta{L}{s}$, for special cases where $\Cl_F$ has exponent $2$.

\begin{prop} \label{prop:ex2Cl=0}
Suppose that $F$ is an imaginary quadratic field where $\Cl_F$ is isomorphic to $(\Z/2)^e$ for some $e$. Then for any fractional ideal $I$ of $F$, we have
\[
\mathrm{Rot}_I(s) = \frac{1}{2} \slzeta{I}{s} \quad \text{and} \quad \sum_{n=1}^{\infty} \frac{ \# \{ [J]: N(J)=n\} }{n^{s}}= \zeta_F(s) \prod_{p} (1-p^{-s}).
\]    
\end{prop}

\begin{proof}
Noting that $J\overline{J} = (N(J))$ for any fractional ideal $J$, $[\overline{J}]=[J]^{-1}$ in $\Cl_F$. Since we are assuming $\Cl_F \simeq (\Z/2)^e$, we have $[\overline{J}]=[J]$. Thus $Z_{I,1}(s) = \frac{1}{2} Z_{I,1}^{+}(s)$ by Proposition \ref{prop:Z1}. By Proposition \ref{prop:zeta=sumZ} and \eqref{eqn:def-Rot-Refl}, we have $\mathrm{Rot}_I(s) = \frac{1}{2} \slzeta{I}{s}$.

Let us show that $J$ is principal for any fractional ideal $J$ of $F$ with $N(J)=1$. Considering prime ideal factorization of such a $J$, one may observe that the exponents of $\mathfrak{p}$ and $\overline{\mathfrak{p}}$ for each split factor $\mathfrak{p}$ of $\mathcal{O}_F$ should add up to zero, and the exponents for inert and ramified factors should all be zero. Hence $J$ is a product of $(\mathfrak{p}/\overline{\mathfrak{p}})^{k_p}$ for some split prime ideals $\mathfrak{p}$ and $k_p\in\z$. Since $\mathfrak{p}/\overline{\mathfrak{p}}$ is principal, so is $J$.

Note that $N(J_1)=N(J_2)=n$ implies that $J_1 J_2^{-1}$ is principal, and hence $[J_1]=[J_2]$. Thus combining this with \eqref{eqn:sumover-mathcal(N)(R)} (note that we do not need the assumption that $\of{}$ is a PID to verify \eqref{eqn:sumover-mathcal(N)(R)}), we have 
\[
\sum_{n=1}^{\infty} \frac{ \# \{ [J]: N(J)=n\} }{n^{s}} = \sum_{n \in \mathcal{N}(\of{})} n^{-s} = \zeta_F(s) \prod_{p} (1-p^{-s}),
\]
where $\mathcal{N}(\of{})=\{N(J) : J\le \of{}\}$. This completes the proof of the proposition.
\end{proof}

\subsection{When $D=-3$}\label{subsec:hexagon}
In this case, we have $|\of{\times}|=6$, $|\Cl_F|=1$, and the corresponding binary $\z$-lattice $L$ satisfies $Q_L\cong x^2+xy+y^2$. The only prime ramified in $F$ is $r=3$, and all the split/inert primes are characterized as $p \equiv 1 \Mod{3}$ and $q \equiv 2 \Mod{3}$, respectively. Combining Theorem \ref{thm:sl.final} and Proposition \ref{prop:ex2Cl=0} gives
\[
\slzeta{x^2+xy+y^2}{s} = \frac{1}{3} \zeta(s) \zeta(s-1) \prod_{p \equiv 1 \Mod{3}} (1-p^{-s}) + \frac{2}{3} \zeta_F(s) \prod_{p \equiv 1 \Mod{3}} (1-p^{-s}), 
\]
or equivalently,
\[
\slzeta{x^2+xy+y^2}{s} = \frac{1}{3} \zeta(s) \zeta(s-1) \prod_{p \equiv 1 \Mod{3}} (1-p^{-s}) + \frac{2}{3} \zeta(s) \prod_{q \equiv 2 \Mod{3}} (1+q^{-s})^{-1}.
\]
For $\glzeta{I}{s} = \mathrm{Rot}_I(s)+\mathrm{Refl}_I(s)$, we first note $\mathrm{Rot}_I(s) = \slzeta{I}{s}/2$ from Proposition \ref{prop:ex2Cl=0}. Thus we only need to add $\mathrm{Refl}_I(s)$, which is given in Proposition \ref{prop:refTermExceptional} for this case. We have
\begin{align*}
\glzeta{x^2+xy+y^2}{s} =& \frac{1}{2}\slzeta{x^2+xy+y^2}{s} + \frac{1}{2} (1-2^{-s}+2 \cdot 2^{-2s}) \zeta(s)^2 \prod_{p \equiv 1 \Mod{3}} (1+p^{-s}) \\
=& \frac{1}{6} \zeta(s) \zeta(s-1) \prod_{p \equiv 1 \Mod{3}} (1-p^{-s}) + \frac{1}{3} \zeta(s) \prod_{q \equiv 2 \Mod{3}} (1+q^{-s})^{-1} \\
&+ \frac{1}{2} (1-2^{-s}+2 \cdot 2^{-2s}) \zeta(s)^2 \prod_{p \equiv 1 \Mod{3}} (1+p^{-s}).
\end{align*}

\subsection{When $D=-4$}\label{subsec:square}
In this case, we have $|\of{\times}|=4$, $|\Cl_F|=1$, and the corresponding binary $\z$-lattice $L$ satisfies $Q_L\cong x^2+y^2$. The only prime ramified in $F$ is $r=2$, and split/inert primes are characterized as $p \equiv 1 \Mod{4}$ and $q \equiv 3 \Mod{4}$ respectively. Combining Theorem \ref{thm:sl.final} and Proposition \ref{prop:ex2Cl=0} gives
\[
\slzeta{x^2+y^2}{s} = \frac{1}{2} \zeta(s) \zeta(s-1) \prod_{p \equiv 1 \Mod{4}} (1-p^{-s}) + \frac{1}{2} \zeta_F(s) \prod_{p \equiv 1 \Mod{4}} (1-p^{-s}) 
\]
or equivalently
\[
\slzeta{x^2+y^2}{s} = \frac{1}{2} \zeta(s) \zeta(s-1) \prod_{p \equiv 1 \Mod{4}} (1-p^{-s}) + \frac{1}{2} \zeta(s) \prod_{q \equiv 3 \Mod{4}} (1+q^{-s})^{-1}.
\]
For $\glzeta{I}{s}$, similar to the case $D=-3$, we have $\mathrm{Rot}_I(s) = \slzeta{I}{s}/2$ from Proposition \ref{prop:ex2Cl=0} and $\mathrm{Refl}_I(s)$ is given in Proposition \ref{prop:refTermExceptional}. Thus we have
\begin{align*}
\glzeta{x^2+y^2}{s} &= \frac{1}{2}\slzeta{x^2+y^2}{s} + \frac{1}{2} (1+2^{-2s}) \zeta(s)^2 \prod_{p \equiv 1 \Mod{4}} (1+p^{-s}) \\
&= \frac{1}{4} \zeta(s) \zeta(s-1) \prod_{p \equiv 1 \Mod{4}} (1-p^{-s}) + \frac{1}{4} \zeta(s) \prod_{q \equiv 3 \Mod{4}} (1+q^{-s})^{-1} \\
&+ \frac{1}{2} (1+2^{-2s}) \zeta(s)^2 \prod_{p \equiv 1 \Mod{4}} (1+p^{-s}).
\end{align*}

\subsection{When $D=-7$}
In this case, we have $|\of{\times}|=2$, $|\Cl_F|=1$, and the corresponding binary $\z$-lattice $L$ satisfies $Q_L\cong x^2+xy+2y^2$. The only prime ramified in $F$ is $r=7$, and split/inert primes are characterized as $p \equiv 1, 2, 4 \Mod{7}$ and $q \equiv 3, 5, 6 \Mod{7}$ respectively. Combining Theorem \ref{thm:sl.final} and Proposition \ref{prop:ex2Cl=0} gives
\[
\slzeta{x^2+xy+2y^2}{s} = \zeta(s)\zeta(s-1) \prod_{p \equiv 1, 2, 4 \Mod 7}(1-p^{-s}).
\]
For $\glzeta{I}{s} = \mathrm{Rot}_I(s)+\mathrm{Refl}_I(s)$, again $\mathrm{Rot}_I(s) = \slzeta{I}{s}/2$ from Proposition \ref{prop:ex2Cl=0}. This time $\mathrm{Refl}_I(s)$ is given in Theorem \ref{thm:refTermGeneral} as
\[
\frac{(1-2^{-s} + 2 \cdot 2^{-2s}) \zeta(s)^2}{2\zeta(2s) \prod_{r} (1+r^{-s})} \left( \sum_{n=1}^{\infty} \frac{ \# \{[J] : |I:J|=n, [J] = [\overline{J}] \} } {n^s} \right).
\]
The number $\# \{[J] : |I:J|=n, [J] = [\overline{J}] \}$ is the same as $\#\{[J] : |I:J|=n \}$ since $[J] = [\overline{J}]$ always holds. Thus by Proposition \ref{prop:ex2Cl=0} we have
\[
\sum_{n=1}^{\infty} \frac{ \# \{[J] : |I:J|=n, [J] = [\overline{J}] \} } {n^s} = \zeta_F(s) \prod_{p} (1-p^{-s})= \zeta(2s)\prod_p(1+p^{-s})\prod_r (1+r^{-s})
\]
and hence
\[
\mathrm{Refl}_I(s)=\frac{1}{2}(1-2^{-s} + 2 \cdot 2^{-2s}) \zeta(s)^2 \prod_{p} (1+p^{-s}).
\]
It follows that
\[
\glzeta{x^2+xy+2y^2}{s} = \frac{1}{2} \zeta(s)\zeta(s-1) \prod_{p \equiv 1, 2, 4 \Mod{7}}(1-p^{-s}) + \frac{1}{2}(1-2^{-s} + 2 \cdot 2^{-2s}) \zeta(s)^2 \prod_{p \equiv 1,2,4 \Mod{7}} (1+p^{-s}).
\]

\subsection{When $D=-8$}
In this case, we have $|\of{\times}|=2$, $|\Cl_F|=1$, and the corresponding binary $\z$-lattice $L$ satisfies $Q_L\cong x^2+2y^2$. The only prime ramified in $F$ is $r=2$, and split/inert primes are characterized as $p \equiv 1, 3 \Mod{8}$ and $q \equiv 5, 7 \Mod{8}$ respectively. Combining Theorem \ref{thm:sl.final} and Proposition \ref{prop:ex2Cl=0} gives
\[
\slzeta{x^2+2y^2}{s} = \zeta(s)\zeta(s-1) \prod_{p \equiv 1, 3 \Mod{8}}(1-p^{-s}).
\] 
For $\glzeta{I}{s} = \mathrm{Rot}_I(s)+\mathrm{Refl}_I(s)$, again $\mathrm{Rot}_I(s) = \slzeta{I}{s}/2$ from Proposition \ref{prop:ex2Cl=0}, and $\mathrm{Refl}_I(s)$ is given in Theorem \ref{thm:refTermGeneral} as
\begin{align*}
\mathrm{Refl}_I(s) &= \frac{(1-2^{-s} + 2 \cdot 2^{-2s}) \zeta(s)^2}{2\zeta(2s) \prod_{r} (1+r^{-s})} \left( \sum_{n=1}^{\infty} \frac{ \# \{[J] : |I:J|=n, [J] = [\overline{J}] \} } {n^s} \right) \\
&+ \frac{(2^{-s}-2^{-2s})\zeta(s)^2}{\zeta(2s) \prod_{r}(1+r^{-s})} \left( \sum_{n=1}^{\infty} \frac{ \# \{[J] \in [\mathscr{I}_{\mathrm{ram}}] : |I:J|=n \}}{n^{s}} \right).
\end{align*}
Since we have $[\mathscr{I}_{\mathrm{ram}}] = \Cl_F$ and $[J] = [\overline{J}]$ for any ideal $J$,  both series are given as $\zeta_F(s) \prod_p (1-p^{-s})$ by Proposition \ref{prop:ex2Cl=0}. Thus $\mathrm{Refl}_I$ is the same as
\[
\frac{(1-2^{-s} + 2 \cdot 2^{-2s} + 2(2^{-s}-2^{-2s})) \zeta(s)^2}{2\zeta(2s) \prod_{r} (1+r^{-s})} \zeta_F(s) \prod_{p}(1-p^{-s}) = \frac{1}{2}(1+2^{-s})\zeta(s)^2 \prod_{p} (1+p^{-s}),
\]
and we have 
\[
\glzeta{x^2+2y^2}{s} = \frac{1}{2}\zeta(s)\zeta(s-1) \prod_{p \equiv 1, 3 \Mod{8}}(1-p^{-s}) + \frac{1}{2}(1+2^{-s})\zeta(s)^2 \prod_{p \equiv 1,3 \Mod{8}} (1+p^{-s}).
\]

\subsection{When $D=-20$}
In this case, we have $|\of{\times}|=2$, $\Cl_F=\{1,-1\}$ and the binary $\z$-lattice $L$ satisfying either $Q_L\cong x^2+5y^2$ or $Q_{L}\cong 2x^2+2xy+3y^2$, corresponds to $1$ or $-1$ in $\Cl_F$, respectively. Primes ramified in $F$ are $2$ and $5$, and split/inert primes are characterized as $p \equiv \{1, 3, 7, 9\} \Mod{20}$ and $q \equiv \{11, 13, 17, 19\} \Mod{20}$, respectively. Combining Theorem \ref{thm:sl.final} and Proposition \ref{prop:ex2Cl=0} gives
\[
\slzeta{x^2+5y^2}{s} = \slzeta{2x^2+2xy+3y^2}{s} = \zeta(s)\zeta(s-1) \prod_{p \equiv 1,3,7,9 \Mod{20}}(1-p^{-s}).
\]
For $\glzeta{I}{s}$, $\mathrm{Rot}_I(s) = \slzeta{I}{s}/2$  from Proposition \ref{prop:ex2Cl=0}. Before considering $\mathrm{Refl}_I(s)$ as given in 
 Theorem \ref{thm:refTermGeneral}, we consider the following series defined for each class $g\in\Cl_F$:
\begin{equation}\label{eqn:defofmathcalL_g}
    \mathcal{L}_g(s) = \sum_{n=1}^{\infty}\frac{i_g(n)}{n^s},    
\end{equation}
where
\begin{equation}\label{eqn:def-i_g}
i_g(n) = \begin{cases}1 & \text{if there is an ideal }J \le R \text{ of norm } n \text{ having class }g, \\ 0 & \text{otherwise}. \end{cases}
\end{equation}

We partition the set of split primes $p$ into two sets according to whether split factor of $p$ is principal or not. Denote $\mathcal{P}_{sp,1}$ and $\mathcal{P}_{sp,2}$ as the sets of split primes $p = \mathfrak{p} \overline{\mathfrak{p}}$ such that $\mathfrak{p}$ is principal/nonprincipal respectively. In this case, since $\q(\sqrt{-1},\sqrt{-5})$ is the Hilbert class field of $F=\q(\sqrt{-5})$, one may verify that $\mathcal{P}_{sp,1} = \{ p: p \equiv 1,9 \Mod{20} \}$ and $\mathcal{P}_{sp,2} = \{p : p \equiv 3,7 \Mod{20} \}$. Write $p_1$ and $p_2$ as primes in $\mathcal{P}_{sp,1}$ and $\mathcal{P}_{sp,2}$ respectively.

\begin{prop} \label{prop:C2ExampleSeries}
We have
\[
\mathcal{L}_1(s) + \mathcal{L}_{-1}(s) = \zeta_F(s) \prod_{p}(1-p^{-s}) = \prod_{p}(1-p^{-s})^{-1} \prod_{q} (1-q^{-2s})^{-1} \prod_{r} (1-r^{-s})^{-1}
\]
and
\[
\mathcal{L}_1(s) - \mathcal{L}_{-1}(s)= (1+2^{-s})^{-1}(1-5^{-s})^{-1} \prod_{p_1 \in \mathcal{P}_{sp,1}} (1-p_1^{-s})^{-1} \prod_{p_2 \in \mathcal{P}_{sp,2}} (1+p_2^{-s})^{-1} \prod_{q}(1-q^{-2s})^{-1} .
\]
\end{prop}

\begin{proof}
The first part follows from observing that $i_1(n)+i_{-1}(n) =\# \{ [J]: N(J)=n\}$, so by Proposition \ref{prop:ex2Cl=0} we have
\[
\mathcal{L}_1(s) + \mathcal{L}_{-1}(s)= \sum_{n=1}^{\infty} \frac{ \# \{ [J]: N(J)=n\} }{n^{s}} = \zeta_F(s) \prod_{p}(1-p^{-s}) = \prod_{p}(1-p^{-s})^{-1} \prod_{q} (1-q^{-2s})^{-1} \prod_{r} (1-r^{-s})^{-1}.
\]
As all Dirichlet coefficients of the right hand side is either $0$ or $1$, it follows that $i_1(n)=i_{-1}(n)=1$ cannot happen.

By considering all prime ideals of $\of{} = \Z[\sqrt{-5}]$, we can observe that a prime is nonprincipal if and only if its norm is in $\{2\} \cup \mathcal{P}_{sp,2}$. Note that an ideal $J$ is principal if and only if it contains even number of nonprincipal prime ideal factors, or equivalently $\nu_2(N(J)) + \sum_{p_2} \nu_{p_2}(N(J))$ is even, where $\nu_p(n)$ denotes the exponent at a rational prime $p$ of $n\in\z$. In the series
\[
 (1+2^{-s})^{-1}(1-5^{-s})^{-1} \prod_{p_1 \in\mathcal{P}_{sp,1}} (1-p_1^{-s})^{-1} \prod_{p_2 \in \mathcal{P}_{sp,2}} (1+p_2^{-s})^{-1} \prod_{q}(1-q^{-2s})^{-1},
\]
its Dirichlet coefficient of $n^{-s}$ is $1$ or $-1$ depending on whether $\nu_2(N(J)) + \sum_{p_2} \nu_{p_2}(N(J))$ is even or odd provided that it is nonzero, so it is same as $i_1(n) - i_{-1}(n)$.
\end{proof}

Now we can express $\mathrm{Refl}_I(s)$ given in Theorem \ref{thm:refTermGeneral}
\begin{align*}
\mathrm{Refl}_I(s) &= \frac{(1-2^{-s} + 2 \cdot 2^{-2s}) \zeta(s)^2}{2\zeta(2s) \prod_{r} (1+r^{-s})} \left( \sum_{n=1}^{\infty} \frac{ \# \{[J] : |I:J|=n, [J] = [\overline{J}] \} } {n^s} \right) \\
&+ \frac{2^{-s}\zeta(s)^2}{\zeta(2s) \prod_{r}(1+r^{-s})} \left( \sum_{n=1}^{\infty} \frac{ \# \{[J] \in [\mathscr{I}_{\mathrm{ortho}}] : |I:J|=n \}}{n^{s}} \right) \\
&-  \frac{2^{-2s}\zeta(s)^2}{\zeta(2s) \prod_{r}(1+r^{-s})} \left( \sum_{n=1}^{\infty} \frac{ \# \{[J] \in [\mathscr{I}_{\mathrm{ram}}] \setminus [\mathscr{I}_{\mathrm{ortho}}] : |I:J|=n \}}{n^{s}} \right)
\end{align*}
using $\mathcal{L}_1(s)$ and $\mathcal{L}_{-1}(s)$. Since $[J]=[\overline{J}]$ for any $J$, we have
\[
\sum_{n=1}^{\infty} \frac{ \# \{[J] : |I:J|=n, [J] = [\overline{J}] \} } {n^s} = \sum_{n=1}^{\infty}  \frac{ \# \{[J] : |I:J|=n \} } {n^s} = \mathcal{L}_1(s) + \mathcal{L}_{-1}(s).
\]
The subgroup $\mathscr{I}_{\mathrm{ortho}}$ is spanned by the class of the ideal $(\sqrt{-5})$, so it consists of principal class. The subgroup $\mathscr{I}_{\mathrm{ram}}$ is spanned by the ideals $(\sqrt{-5})$ and $(2,1+\sqrt{-5})$, so it is the whole $\Cl_F$. 
 It follows that
\begin{align*}
\sum_{n=1}^{\infty}  \frac{ \# \{[J] \in [\mathscr{I}_{\mathrm{ortho}}] : |I:J|=n \} }{n^{s}} &= \begin{cases} \mathcal{L}_1(s) & \text{if } [I] \text{ trivial}, \\ \mathcal{L}_{-1}(s) & \text{if } [I] \text{ nontrivial} \end{cases}  \\
&= \frac{1}{2} ( \mathcal{L}_{1}(s) + \mathcal{L}_{-1}(s) ) + \frac{\chi(I)}{2} ( \mathcal{L}_{1}(s) - \mathcal{L}_{-1}(s) )
\end{align*}
and
\begin{align*}
\sum_{n=1}^{\infty}  \frac{ \# \{[J] \in [\mathscr{I}_{\mathrm{ram}}] \setminus [\mathscr{I}_{\mathrm{ortho}}] : |I:J|=n \} }{n^{s}} &= \begin{cases} \mathcal{L}_{-1}(s) & \text{if } [I] \text{ trivial} \\ \mathcal{L}_{1}(s) & \text{if } [I] \text{ nontrivial} \end{cases} \\
&= \frac{1}{2} ( \mathcal{L}_{1}(s) + \mathcal{L}_{-1}(s) ) - \frac{\chi(I)}{2} ( \mathcal{L}_{1}(s) - \mathcal{L}_{-1}(s) )
\end{align*}
where $\chi(I)=1$ if $[I]=1$ and $\chi(I)=-1$ if $[I]=-1$.
It follows that the whole expression for $\mathrm{Refl}_I(s)$ is given as
\[
\mathrm{Refl}_I(s) = \frac{(1+2^{-2s})\zeta(s)^2}{2 \zeta(2s) \prod_{r}(1+r^{-s})}(\mathcal{L}_1(s)+\mathcal{L}_{-1}(s)) + \chi(I) \frac{(2^{-s}+2^{-2s})\zeta(s)^2}{2 \zeta(2s) \prod_{r}(1+r^{-s})}(\mathcal{L}_1(s)-\mathcal{L}_{-1}(s)).
\]
By expanding out $\mathcal{L}_1(s)+\mathcal{L}_{-1}(s)$ and $\mathcal{L}_1(s)-\mathcal{L}_{-1}(s)$ as of Proposition \ref{prop:C2ExampleSeries}, one can further reduce the expression as
\[
\frac{(1+2^{-2s})\zeta(s)^2}{2 \zeta(2s) \prod_{r}(1+r^{-s})}(\mathcal{L}_1(s)+\mathcal{L}_{-1}(s)) =\frac{1}{2}(1+2^{-2s})\zeta(s)^2 \prod_{p} (1+p^{-s})
\]
and
\[
\frac{(2^{-s}+2^{-2s})\zeta(s)^2}{2 \zeta(2s) \prod_{r}(1+r^{-s})}(\mathcal{L}_1(s)-\mathcal{L}_{-1}(s)) =\frac{1}{2}(2^{-s}-2^{-2s})\zeta(s)^2 \prod_{p_1} (1+p_1^{-s}) \prod_{p_2} (1-p_2^{-s}).
\]

Putting everything together gives
\begin{align*}
\glzeta{x^2+5y^2}{s} &= \frac{1}{2} \zeta(s) \zeta(s-1) \prod_{p \equiv 1,3,7,9 \Mod{20}} (1-p^{-s})+ \frac{1}{2}(1+2^{-2s})\zeta(s)^2 \prod_{p \equiv 1,3,7,9 \Mod{20}} (1+p^{-s}) \\
&+ \frac{1}{2}(2^{-s}-2^{-2s})\zeta(s)^2 \prod_{p_1 \equiv 1,9 \Mod{20}} (1+p_1^{-s}) \prod_{p_2 \equiv 3,7 \Mod{20}} (1-p_2^{-s}).
\end{align*}
and
\begin{align*}
\glzeta{2x^2+2xy+3y^2}{s} &= \frac{1}{2} \zeta(s) \zeta(s-1) \prod_{p \equiv 1,3,7,9 \Mod{20}} (1-p^{-s}) + \frac{1}{2}(1+2^{-2s})\zeta(s)^2 \prod_{p \equiv 1,3,7,9 \Mod{20}} (1+p^{-s}) \\
&- \frac{1}{2}(2^{-s}-2^{-2s})\zeta(s)^2 \prod_{p_1 \equiv 1,9 \Mod{20}} (1+p_1^{-s}) \prod_{p_2 \equiv 3,7 \Mod{20}} (1-p_2^{-s}).
\end{align*}

\subsection{When $D=-23$} 
In this case, we have $|\of{\times}|=2$, $\Cl_F=\{1,\omega,\omega^2\}$, and the binary $\z$-lattice $L$ satisfying $Q_L\cong x^2+xy+6y^2$, $Q_L\cong2x^2+xy+3y^2$ or $Q_L\cong2x^2-xy+3y^2$ corresponds to either $1$, $\omega$ or $\omega^2$ in $\Cl_F$, respectively. The ring of integers $R = \Z\left[\frac{1 +\sqrt{-23}}{2}\right]$ of $F$ has only one ramified prime $r=23=\left(\sqrt{-23}\right)^2$, and split/inert primes are characterized as $\left( \frac{p}{23} \right)=1$ and $\left( \frac{q}{23} \right)=-1$ respectively. We denote $\mathcal{P}_{sp,1}$ and $\mathcal{P}_{sp,2}$ as sets of splitting primes $p$ such that one of its split factors $\mathfrak{p}$ (hence both of them) ($p = \mathfrak{p}\overline{\mathfrak{p}}$) is principal/nonprincipal in $\Cl_F$ respectively.

\begin{rmk}
One can show that $p \in \mathcal{P}_{sp,1}$ if and only if $x^3-x-1$ has three distinct solutions mod $p$, by utilizing that $L = F(x)/(x^3-x-1)$ is the Hilbert class field of $F$. This set is not described in terms of congruence conditions in $p$, so we will leave as it is.
\end{rmk}

For each class $g\in\Cl_F$, we recall $\mathcal{L}_g(s)$ defined in \eqref{eqn:defofmathcalL_g} and we fix an ideal $I_g$ such that $[I_g]=g$.
Proposition \ref{prop:Z1+} gives $\slzeta{I}{s}$ as follows:
\begin{equation}\label{eqn:slzetaD=-23}
\slzeta{x^2+xy+6y^2}{s} = \slzeta{2x^2+xy+3y^2}{s}=\slzeta{2x^2-xy+3y^2}{s} = \frac{\zeta_{\Z^2}(s)}{\zeta_F(s)} \left( \sum_{n=1}^{\infty} \frac{ \# \{ [J]: N(J)=n\} }{n^{s}} \right).
\end{equation}
The series inside can be written as
\begin{equation}\label{eqn:sumnormnD=-23}
\left( \sum_{n=1}^{\infty} \frac{ \# \{ [J]: N(J)=n\} }{n^{s}} \right) = \mathcal{L}_1(s) +\mathcal{L}_{\omega}(s) + \mathcal{L}_{\omega^2}(s),
\end{equation}
so calculating all $\mathcal{L}_g(s)$ will give  $\slzeta{I}{s}$. For $\glzeta{I}{s}$, we also need the series of the form
\[
\mathcal{E}_{g}(s) := \sum_{n=1}^{\infty} \frac{ \# \left( \{[J]:|I_g:J|=n\} \} \setminus \{[\overline{J}] : |I_g:J|=n \right)}{n^s}
\]
to describe $\mathrm{Rot}_I(s)$ as
\[
\mathrm{Rot}_I(s)= \frac{1}{2} \slzeta{I}{s} + \frac{\zeta_{\Z^2}(s)}{2\zeta_{F}(s)} \mathcal{E}_{[I]}(s).
\]

\begin{prop} \label{prop:C3ExampleSeries}
We have
\[
\mathcal{L}_1(s) = \left( \prod_{p_2 \in \mathcal{P}_{sp,2}}(1-p_2^{-s})^{-1} - \sum_{p_2 \in \mathcal{P}_{sp,2}} p_2^{-s} \right)  \prod_{p_1\in \mathcal{P}_{sp,1}} (1-p_1^{-s})^{-1} \prod_q (1-q^{-2s})^{-1}\prod_r (1-r^{-s})^{-1}
\]
and
\[
\mathcal{L}_{\omega}(s) = \mathcal{L}_{\omega^2}(s) = \left( \prod_{p_2 \in \mathcal{P}_{sp,2}}(1-p_2^{-s})^{-1} - 1 \right) \prod_{p_1 \in \mathcal{P}_{sp,1}}(1-p_1^{-s})^{-1} \prod_q(1-q^{-2s})^{-1} \prod_r (1-r^{-s})^{-1} .
\]
We also have $\mathcal{E}_{1}(s) =0$ and
\[
\mathcal{E}_{\omega}(s) = \mathcal{E}_{\omega^2}(s) = \left( 1+ \sum_{p \in \mathcal{P}_{sp,2}}p_2^{-s} \right) \prod_{p_1 \in \mathcal{P}_{sp,1}} (1-p_1^{-s})^{-1} \prod_q (1-q^{-2s})^{-1} \prod_r (1-r^{-s})^{-1}.
\]
\end{prop}

\begin{proof}
For each $g \in \Cl_F$ and $n\in\n$, let $C_{g,n}=\{[J] : |I_g : J|=n\}$ and $\overline{C_{g,n}}=\{[\overline{J}] : |I_g : J|=n\}$. Observe that $C_{g,n}$ determines $\mathcal{L}_g(s)$ and $\mathcal{E}_g(s)$, namely, recalling the definition of $i_g(n)$ in \eqref{eqn:def-i_g}, we have $i_g(n)=1$ if $g \in C_{1,n}$ and $i_g(n)=0$ otherwise; and
\[
\mathcal{E}_g(s) = \sum_{n=1}^{\infty} \frac{ |C_{g,n} \setminus \overline{C_{g,n}}|}{n^s}.
\]
Therefore, we now aim to describe the set $C_{g,n}$ for each $g\in \Cl_F$ concretely. We also note that
\[
C_{g,n}=\left\{[J] : |I_g : J|=n\right\}=\left\{[I_g][I_g^{-1}J] : |R : I_g^{-1}J|=n\right\} = \left\{g[J] : |R : J|=n\right\} = \left\{g h : h\in C_{1,n}\right\}.
\]
Hence it suffices to determine the set $C_{1,n}$.

To construct an ideal $J$ of $R$ of norm $n$, one needs to choose its prime ideal factors so that the multiplication of their norm becomes $n$. Norms of prime ideals of $R$ consist of split primes $p$ and ramified primes $r$, and $q^2 = N(q)$ for inert primes $q$. An ideal $J$ of norm $n$ exists if $n$ can be represented as a product of such numbers. Denote the set $\prod(p_1, p_2, q^2, r)$ as the set of numbers represented as products of split primes $p$, ramified primes $r$, and squares of inert primes $q^2$. Then we have $C_{1,n} \neq \emptyset$ if and only if $n \in \prod(p_1, p_2, q^2, r)$.

For such $n \in \prod(p_1, p_2, q^2, r)$, the ideal class $[J]$ of $J$ where $N(J)=n$ is determined by the choice of prime factors of $J$ such that the products of their norms become $n$. Since the ideals of norm $q^2$ and $r$ are given uniquely as principal ideals $(q)$ and $\gamma$ (the principality of $\gamma$ follows from the fact that $\gamma^2 = (r)$ and $\Cl_F \simeq C_3$),  they do not affect the ideal class of $J$. For $p_1 \in \mathcal{P}_{sp,1}$, its two split factors are also all principal, so they also do not affect the ideal class of $J$ either. This shows that the ideal class of $J$ only depends on the choice of prime ideal factors of norm $p_2$ for $p_2 \in \mathcal{P}_{sp,2}$. It follows that if we denote $\prod(p_1, q^2, r)$ as the set of numbers represented as the products of $p_1 \in\mathcal{P}_{sp,1}$, $q^2$, and $r$, then $C_{1,n}$ does not change if we multiply $n$ by an element in $\prod(p_1, q^2, r)$.

Thus we only need to consider when $n$ is a product of primes $p_2$ in $\mathcal{P}_{sp,2}$ If $n=1$ then $C_{1,n}=\{1\}$ as the only ideal of norm $1$ is $R$. If $n=p_2$ for some $p_2 \in \mathcal{P}_{sp,2}$, then $C_{1,n} = \{[\mathfrak{p}],[\overline{\mathfrak{p}}]\}=\{\omega,\omega^2\}$ for $p_2 = \mathfrak{p}\overline{\mathfrak{p}}$. We show that $C_{1,n} = \{1,\omega, \omega^2\}$ for all other cases. Write $n = p_2 p_2' m$ for $p_2, p_2' \in \mathcal{P}_{sp,2}$, either same or different. Let $\mathfrak{p}$ and $\mathfrak{p}'$ be one of the split factors of $(p_2)$ and $(p_2')$ respectively so that their ideal class in $\Cl_F$ is both $w$. Then $\mathfrak{p} \mathfrak{p}', \mathfrak{p} \overline{\mathfrak{p'}}, \overline{\mathfrak{p}} \overline{\mathfrak{p'}}$ have classes $w^2, 1, w$ in $\Cl_F$, respectively. Thus for any ideal class $g \in \Cl_F$ one can construct an ideal $J$ such that $N(J)=n$ and $[J]=g$ by constructing any ideal $J'$ having norm $m$ and multiplying $J'$ with $\mathfrak{p} \mathfrak{p}', \mathfrak{p} \overline{\mathfrak{p'}}, \overline{\mathfrak{p}} \overline{\mathfrak{p'}}$ to make it into desired class. 

By combining everything, we obtain the following complete descriptions of $C_{1,n}$:
\begin{equation}\label{eqn:C_1,n}
C_{1,n} = \begin{cases}
\emptyset & \text{if } n \notin \prod(p_1, p_2, q^2, r), \\
1 & \text{if } n \in \prod(p_1, q^2, r), \\
\{\omega, \omega^2\} & \text{if } n \in p_2 \cdot \prod(p_1, q^2, r) \text{ for some } p_2 \in \mathcal{P}_{sp,2}, \\
\{1,\omega, \omega^2\} & \text{if } n \in \prod(p_1, p_2, q^2, r) \text{ and not the above cases}.
\end{cases}
\end{equation}

We are now ready to calculate $\mathcal{L}_g$ and $\mathcal{E}_g$ using \eqref{eqn:C_1,n}. As $i_g(n)=1$ if and only if $g \in C_{1,n}$, we have
\[
i_1(n) = \begin{cases}
0 & \text{if } n \in p_2 \cdot \prod(p_1, q^2, r) \text{ for some } p_2 \in \mathcal{P}_{sp,2} \text{ or } n \notin \prod(p_1, p_2, q^2, r),\\
1 & \text{otherwise},
\end{cases}
\]
and
\[
i_{\omega}(n) = i_{\omega^2}(n) = \begin{cases}
0 & \text{if } n \in \prod(p_1, q^2, r) \text{ or } n \notin \prod(p_1, p_2, q^2, r) ,\\ 1 & \text{otherwise}.
\end{cases}
\]
This allows us to obtain $\mathcal{L}_g(s)$.

For $\mathcal{E}_g(s)$, we note that it is given as a Dirichlet series for $|C_{g,n} \setminus \overline{C_{g,n}}|$. Since $C_{g,n} = g C_{1,n}$ and $\overline{C_{1,n}}= C_{1,n}$, we have $C_{g,n} \setminus \overline{C_{g,n}} = g C_{1,n} \setminus \overline{g} C_{1,n}$. Thus if $g=1$ then $C_{g,n} \setminus \overline{C_{g,n}} = \emptyset$. Hence $\mathcal{E}_1(s)=0$. On the other hand, for $g=\omega, \omega^2$ we have
\[
|C_{\omega,n} \setminus \overline{C_{\omega,n}}| = |C_{\omega^2,n} \setminus \overline{C_{\omega^2,n}}|= \begin{cases} 1 & \text{if } n \in \prod(p_1, q^2, r) \cup \bigcup_{p_2 \in \mathcal{P}_{sp,2}} p_2 \prod(p_1, q^2, r), \\ 0 & \text{otherwise},
\end{cases}
\]
which gives $\mathcal{E}_{\omega}(s) = \mathcal{E}_{\omega^2}(s)$ as required.
\end{proof}

Combining \eqref{eqn:slzetaD=-23}, \eqref{eqn:sumnormnD=-23}, Proposition \ref{prop:C3ExampleSeries}, $\zeta_{\z^2}(s)=\zeta(s)\zeta(s-1)$ and
\[
\frac{\prod_{p_1 \in \mathcal{P}_{sp,1}}(1-p_1^{-s})^{-1} \prod_{q}(1-q^{-2s})^{-1} \prod_{r}(1-r^{-s})^{-1}}{\zeta_F(s)} = \prod_{p_1 \in \mathcal{P}_{sp,1}}(1-p_1^{-s}) \prod_{p_2 \in \mathcal{P}_{sp,2}} (1-p_2^{-s})^{2},
\]
we have that all the $\slzeta{x^2+xy+6y^2}{s}$, $\slzeta{2x^2+2xy+3y^2}{s}$, and $\slzeta{2x^2-2xy+3y^2}{s}$ are equal to
\[
\zeta(s)\zeta(s-1) \prod_{p_1 \in \mathcal{P}_{sp,1}} (1-p_1^{-s}) \prod_{p_2 \in \mathcal{P}_{sp,2}} (1-p_2^{-s})^2 \left(3 \prod_{p_2 \in \mathcal{P}_{sp,2}}(1-p_2^{-s})^{-1} - \sum_{p_2 \in \mathcal{P}_{sp,2}} p_2^{-s} - 2 \right).
\]
For $\glzeta{I}{s}$, we need to track how $\mathrm{Rot}_I(s)$ and $\mathrm{Refl}_I(s)$ change by $g=[I]$. For $x^2+xy+6y^2$ (so that $g=1$), we have $\mathcal{E}_g(s)=0$ so $\mathrm{Rot}_I(s)=(1/2) \slzeta{x^2+xy+6y^2}{s}$. According to Theorem \ref{thm:refTermGeneral}, $\mathrm{Refl}_I(s)$ is given as
\[
\frac{(1-2^{-s} + 2 \cdot 2^{-2s}) \zeta(s)^2}{2\zeta(2s) \prod_{r} (1+r^{-s})} \left( \sum_{n=1}^{\infty} \frac{ \# \{[J] : |I:J|=n, [J] = [\overline{J}] \} } {n^s} \right) =  \frac{(1-2^{-s} + 2 \cdot 2^{-2s}) \zeta(s)^2}{2\zeta(2s) \prod_{r} (1+r^{-s})} \mathcal{L}_1(s)
\]
and substituting expression for $\mathcal{L}_1(s)$ of Proposition \ref{prop:C3ExampleSeries} gives
\[
\frac{(1-2^{-s}+2 \cdot 2^{-2s})\zeta(s)^2}{2} \prod_{p_1 \in \mathcal{P}_{sp,1}}(1+p_1^{-s}) \prod_{p_2 \in \mathcal{P}_{sp,2}}(1-p_2^{-2s}) \left( \prod_{p_2 \in \mathcal{P}_{sp,2}} (1-p_2^{-s})^{-1} - \sum_{p_2 \in \mathcal{P}_{sp,2}} p_2^{-s} \right).
\]
Adding $\mathrm{Rot}_I(s)$ and $\mathrm{Refl}_I(s)$ gives
\begin{align*}
&\glzeta{x^2+xy+6y^2}{s} \\
=& \frac{1}{2}\zeta(s)\zeta(s-1) \prod_{p_1 \in \mathcal{P}_{sp,1}} (1-p_1^{-s}) \prod_{p_2 \in \mathcal{P}_{sp,2}} (1-p_2^{-s})^2 \left(3 \prod_{p_2 \in \mathcal{P}_{sp,2}}(1-p_2^{-s})^{-1} - \sum_{p_2 \in \mathcal{P}_{sp,2}} p_2^{-s} - 2 \right) \\
+&\frac{(1-2^{-s}+2 \cdot 2^{-2s})\zeta(s)^2}{2} \prod_{p_1 \in \mathcal{P}_{sp,1}}(1+p_1^{-s}) \prod_{p_2 \in \mathcal{P}_{sp,2}}(1-p_2^{-2s}) \left( \prod_{p_2 \in \mathcal{P}_{sp,2}} (1-p_2^{-s})^{-1} - \sum_{p_2 \in \mathcal{P}_{sp,2}} p_2^{-s} \right).
\end{align*}
For $\glzeta{2x^2+xy+3y^2}{s}$, we add additional $\dfrac{\zeta_{\Z^2}(s)}{2\zeta_{F}(s)} \mathcal{E}_{\omega}(s)$ to $\mathrm{Rot}_I(s)$ which is
\[
\frac{\zeta_{\Z^2}(s)}{2\zeta_{F}(s)} \mathcal{E}_{\omega}(s) =  \frac{1}{2}\zeta(s)\zeta(s-1) \prod_{p_1 \in \mathcal{P}_{sp,1}} (1-p_1^{-s}) \prod_{p_2 \in \mathcal{P}_{sp,2}} (1-p_2^{-s})^2 \left(1+ \sum_{p_2 \in \mathcal{P}_{sp,2}} p_2^{-s} \right)
\]
and $\mathrm{Refl}_I(s)$ is changed to $\dfrac{(1-2^{-s} + 2 \cdot 2^{-2s}) \zeta(s)^2}{2\zeta(2s) \prod_{r} (1+r^{-s})} \mathcal{L}_{\omega}(s)$. Considering all those gives
\begin{align*}
&\glzeta{2x^2+xy+3y^2}{s} \\
=& \frac{1}{2}\zeta(s)\zeta(s-1) \prod_{p_1 \in \mathcal{P}_{sp,1}} (1-p_1^{-s}) \prod_{p_2 \in \mathcal{P}_{sp,2}} (1-p_2^{-s})^2 \left(3 \prod_{p_2 \in \mathcal{P}_{sp,2}}(1-p_2^{-s})^{-1} - 1 \right) \\
+&\frac{(1-2^{-s}+2 \cdot 2^{-2s})\zeta(s)^2}{2} \prod_{p_1 \in \mathcal{P}_{sp,1}}(1+p_1^{-s}) \prod_{p_2 \in \mathcal{P}_{sp,2}}(1-p_2^{-2s}) \left( \prod_{p_2 \in \mathcal{P}_{sp,2}} (1-p_2^{-s})^{-1} -1\right).
\end{align*}
We have $\glzeta{2x^2-xy+3y^2}{s}= \glzeta{2x^2+xy+3y^2}{s}$, as $\mathcal{L}_{\omega}(s) = \mathcal{L}_{\omega^2}(s)$ and $\mathcal{E}_{\omega}(s) = \mathcal{E}_{\omega^2}(s)$.
\begin{rmk}
    We have numerically checked the formulas of $\slzeta{L}{s}$ and $\glzeta{L}{s}$ for each of the above binary lattices $L$ for the first several coefficients; namely, we wrote a computer program based on MAGMA \cite{Magma} computing the number $a_m(L)$ (resp. $a_m^+(L)$) of (resp. properly) isometric sublattice of $L$ of index $m$, and checked it matches with our results for each $m$ up to $300$.
\end{rmk}

\section{Analytic properties of zeta functions of binary lattices}\label{sec:analytic}
In this section we record some analytic properties of zeta functions $\slzeta{L}{s}$ and $\glzeta{L}{s}$ of binary $\z$-lattices $L$.

\begin{thm}\label{thm:analytic.property}
Let $L$ be a binary $\z$-lattice with $-4d_L$ being the discriminant of an imaginary quadratic field. Both $\slzeta{L}{s}$ and $\glzeta{L}{s}$ can be meromorphically continued to $\{ s \in \mathbb{C} | \Re(s) >1 \}$, and they both have a simple pole at $s=2$ in that region. In particular, we have $\alpha_L^{\SL} = \alpha_{L}^{\GL}=2$.
\end{thm}

\begin{proof}
In the region $\{ s \in \mathbb{C} | \Re(s) >1 \}$, we note that $\zeta(s),\zeta_F(s)$, and any products of the form $\prod_{p \in \mathcal{P}} (1 \pm p^{-s})$ for some set $\mathcal{P}$ of primes are all holomorphic and nonvanishing. Also any Dirichlet series of the form
\[
\sum_{n=1}^\infty \frac{\#\{ [J] \in S  : |I:J|=n \}}{n^s}
\]
for some subset $S \subset \Cl_F$ and the series
\[
\sum_{n=1}^{\infty} \frac{\# \{ [J] : |I : J|=n \} \setminus \{ [\overline{J}] : |I : J|=n \} }{n^s}
\]
are holomorphic in $\Re(s) >1$, as those coefficients are all bounded by $|\Cl_F|$. Thus the formula of $\slzeta{L}{s}$ given in Theorem \ref{thm:sl.final}
\[
\slzeta{L}{s} = \frac{2}{|\of{\times}|} \frac{\zeta_{\Z^2}(s)}{\zeta_{F}(s)} \left( \sum_{n=1}^{\infty} \frac{ \# \{ [J]: J \le \of{}, N(J)=n\} }{n^{s}} \right) +  \frac{|\of{\times}| - 2}{|\of{\times}|}  \prod_{p \text{ split}} (1-p^{-s}) \zeta_{F}(s)
\]
is basically of a form $\slzeta{L}{s} = \zeta_{\Z^2}(s) H_1(s) + H_2(s)$ for some function $H_1, H_2$ that is holomorphic in $\Re(s) >1$. As $\zeta_{\Z^2}(s)=\zeta(s)\zeta(s-1)$ has a simple pole at $s=2$ with residue $\zeta(2)$, this gives a meromorphic description of $\slzeta{L}{s}$ in $\Re(s) >1$, and its residue at $s=2$ is
\[
\mathrm{Res}_{s=2} \slzeta{L}{s} = \frac{2\zeta(2)}{|\of{\times}|\zeta_F(2)} \left( \sum_{n=1}^{\infty} \frac{\# \{ [J] : |I : J|=n \}}{n^2}\right) > 0.
\]
On the other hand, recall the following formula of $\glzeta{L}{s}$:
\[
\glzeta{L}{s} = \mathrm{Rot}_L(s) + \mathrm{Refl}_L(s)
\]
where
\[
\mathrm{Rot}_L(s)=
\frac{1}{2} \slzeta{L}{s} + \frac{\zeta_{\Z^2}(s)}{2\zeta_{F}(s)} \sum_{n=1}^{\infty} \frac{\# \left( \{ [J] : |I : J|=n \} \setminus \{ [\overline{J}] : |I : J|=n \} \right)}{n^s}.
\]
One may similarly prove that $\mathrm{Rot}_L(s)$ is meromorphic in $\Re(s)>1$, and has a unique simple pole at $s=2$ with non-negative residue. Moreover, the formula for $\mathrm{Refl}_L(s)$ given in Theorem \ref{thm:refTermGeneral} and Proposition \ref{prop:refTermExceptional} suggests that $\mathrm{Refl}_L(s)$ is holomorphic in $\Re(s)>1$. This completes the proof of the theorem.
\end{proof}

From Theorem \ref{thm:analytic.property} one may observe that in both $\slzeta{L}{s}$ and $\glzeta{L}{s}$, their poles at $s=2$  come from $Z_{L,\pm 1}^{+}$ and $Z_{L, \pm 1}$, and the remaining sums over other nontrivial rotational terms and reflection terms are holomorphic. We conjecture this behavior to be true for $L$ of higher rank $n\geq2$.

\begin{conj}
    Let $L$ be a $\Z$-lattice of rank $n\geq2$ on a positive definite quadratic space $V$. The series  $Z_{L,\pm 1}$ and $Z_{L,\pm 1}^{+}$ can be meromorphically continued in $\Re(s)>n-1$ with a single pole at $s=n$ with positive residue. All other terms $Z_{L, \rho}$, $Z_{L,\rho}^{+}$, and the sum $\sum_{\rho \in O^{-}(V)} Z_{L, \rho}$ are holomorphic in $\Re(s)>n-1$.
\end{conj}

For each prime number $p$, let us define the Dirichlet series enumerating subforms of $p$-power index of $L$ up to (proper) isometry as follows:
\[
    \slzeta{L,p}{s}:=\sum_{i=0}^\infty a_{p^{i}}^+(L) p^{-is} \qquad \text{and} \qquad   \glzeta{L,p}{s}:=\sum_{i=0}^\infty a_{p^{i}}(L) p^{-is}.
\]
According to our computations in Section \ref{sec:examples}, none of our explicit formulas for $\glzeta{L}{s}$ seems to satisfy any form of an Euler product. However, one may check that for $D=-7,-8$, and $-20$, we have
\[\slzeta{L}{s}=\prod_{p\,prime}\slzeta{L,p}{s}.\]   
More generally, we have the following theorem.
\begin{thm}
Let $L$ be a binary $\z$-lattice with $-4d_L$ being the discriminant $D$ of an imaginary quadratic field. Except for the exceptional cases when $D=-3$ and $D=-4$,  $\slzeta{L}{s}$ satisfies the Euler product    
\[\slzeta{L}{s}=\prod_{p\,prime}\slzeta{L,p}{s}\]
if and only if $\Cl_F$ is isomorphic to $(\Z/2)^e$ for some $e\in\N_0$.
\end{thm}

\begin{proof}

In the description of $\slzeta{L}{s}$ given in Theorem \ref{thm:sl.final}, it is clear that $\slzeta{L}{s}$ has an Euler product if and only if the series 
 \[
 \mathcal{L}_F(s) =\sum_{n=1}^{\infty} \frac{ \# \{ [J] : J\le \of{}, \, N(J)=n\} }{n^{s}}
 \]
has an Euler product.

In Proposition \ref{prop:ex2Cl=0} we have shown that if $\Cl_F \simeq (\Z/2)^e$ for some $e$ then $\slzeta{L}{s} = \zeta_{\Z^2}(s) \prod_{p} (1-p^{-s})$, so it suffices to show the other direction that if $\mathcal{L}_F(s)$ has an Euler product then $\Cl_F$ has exponent $2$.

To show $\Cl_F$ has exponent $2$, it suffices to show that all its generators have exponent $1$ or $2$. We know that inert primes $q$ are trivial in $\Cl_F$ and ramified factors $\gamma$ all satisfy $[\gamma]^2=1$. Hence it suffices to show that $[\mathfrak{p}]^2=1$ for any prime ideal $\mathfrak{p}$ over a split prime.

Assume to the contrary that $[\mathfrak{p}_1]^2 \neq 1$ for $p_1 = \mathfrak{p}_1 \overline{\mathfrak{p}_1}$. By Chebotarev density theorem, there exists another prime ideal over a split prime of $F$ having the same class as $\mathfrak{p}_1$, and that ideal should be also a split ideal $\mathfrak{p}_2$. Write $p_2 = \mathfrak{p}_2 \overline{\mathfrak{p}_2}$ for $[\mathfrak{p}_1] = [\mathfrak{p}_2] = g \in \Cl_F$ such that $g^2 \neq 1$.

Now consider the $p_1^{-s},p_2^{-s}$, and $(p_1 p_2)^{-s}$-coefficient of $\mathcal{L}_F(s)$. Ideals of $F$ having norm $p_1$ are $\mathfrak{p}_1$ and $\overline{\mathfrak{p}_1}$, and those two are in different classes. Thus the $p_1^{-s}$-coefficient of $\mathcal{L}_F(s)$ is $2$. Similarly the $p_2^{-s}$- coefficient of $\mathcal{L}_F(s)$ is also $2$. Hence if $\mathcal{L}_F(s)$ has an Euler product then its $(p_1 p_2)^{-s}$-coefficient should be $4$.

Meanwhile, ideals of $F$ having norm $p_1 p_2$ are $\mathfrak{p}_1 \mathfrak{p}_2, \mathfrak{p}_1 \overline{\mathfrak{p}_2}, \overline{\mathfrak{p}_1} \mathfrak{p}_2, \overline{\mathfrak{p}_1 \mathfrak{p}_2}$, and their classes in $\Cl_F$ are $g^2, 1, 1, g^{-2}$ respectively. Thus the $(p_1 p_2)^{-s}$-coefficient of $\mathcal{L}_F(s)$ should be always $\le 3$, and it follows that $\mathcal{L}_F(s)$ cannot have an Euler product.
\end{proof}

Analyzing $\slzeta{L,p}{s}$  also allows us to obtain the following results.

\begin{lem}
Suppose $|\of{\times}|=2$. We have
\[
\slzeta{L,p}{s} = \begin{cases} (1-p^{-s})^{-1} (1-p^{1-s})^{-1} & \text{if }p \text{ inert or ramified} \\
(1-p^{-k_p s})(1-p^{-s})^{-1}(1-p^{1-s})^{-1} & \text{if } p \text{ split} \end{cases}
\]
where $k_p$ is the order of $[\mathfrak{p}]^2$ in $\Cl_F$ for a split factor $\mathfrak{p}$ of $p$.
\end{lem}

\begin{proof}
If $|\of{\times}|=2$, then $\slzeta{L,p}{s}$ comes from the coefficients of $p^{-ks}$ of
\[
\slzeta{L}{s}=\frac{\zeta_{\Z^2}(s)}{\zeta_F(s)} \left( \sum_{n=1}^{\infty} \frac{ \# \{ [J]: J\le \of{}, \, N(J)=n\} }{n^{s}} \right).
\]
Hence it is given as\[
\slzeta{L,p}{s}=\frac{\zeta_p(s)\zeta_p(s-1)}{\zeta_{F,p}(s)} \left( \sum_{m=0}^{\infty} \frac{ \# \{ [J]: J\le \of{}, \, N(J)=p^m\} }{p^{ms}} \right)
\]
where 
\[
\zeta_p(s) = (1 - p^{-s})^{-1} \quad \text{and}\quad \zeta_{F,p}(s) = \begin{cases}
(1-p^{-s})^{-1} & \text{if } p \text{ ramified}, \\
(1-p^{-2s})^{-1} & \text{if } p \text{ inert}, \\
(1-p^{-s})^{-2} & \text{if } p \text{ split}
\end{cases}
\]
are ``$p$-parts'' of $\zeta(s)$ and $\zeta_F(s)$ respectively.\\
{\bf Case 1) $p$ is ramified.\rm} The $m$-th power of the ramified factor over $p$ is the unique ideal of norm $p^m$, so the series
\[
\sum_{m=0}^{\infty} \frac{ \# \{ [J]: J\le \of{}, \, N(J)=p^m\} }{p^{ms}}.
\]
is equal to $(1-p^{-s})^{-1}$. This cancels out with $\zeta_{F,p}(s)$, so $\slzeta{L,p}{s}=\zeta_p(s) \zeta_p(s-1)$.\\
{\bf Case 2) $p$ is inert.\rm} The ideal of norm $p^m$ does not exist when $2 \nmid m$ and exists uniquely as $(p)^{m/2}$ when $2 \vert m$, so the series
\[
\sum_{m=0}^{\infty} \frac{ \# \{ [J]: J\le \of{}, \, N(J)=p^m\} }{p^{ms}}.
\]
is equal to $(1-p^{-2s})^{-1}$. This cancels out with $\zeta_{F,p}(s)$, so $\slzeta{L,p}{s}=\zeta_p(s) \zeta_p(s-1)$.\\
{\bf Case 3) $p$ is split.\rm} Let $(p) = \mathfrak{p} \overline{\mathfrak{p}}$ be the ideal factorization of $p$ in $\of{}$. An ideal of norm $p^m$ is exactly given as $\mathfrak{p}^k \overline{\mathfrak{p}}^{m-k}$ for $0 \le k \le m$, and its ideal class in $\Cl_K$ is $[\mathfrak{p}]^k [\overline{\mathfrak{p}}]^{m-k} = [\mathfrak{p}]^{2k-m}$ (as $[\overline{\mathfrak{p}}] = [\mathfrak{p}]^{-1}$). If $m<k_p$ then those $m+1$ classes are all distinct, and if $m \ge k_p$ then there are exactly $k_p$ classes when $k=0, 1, \cdots, k_{p}-1$. Thus we have
\[
\# \{ [J]: J\le \of{}, \, N(J)=p^m\} = \begin{cases} m+1 & \text{if } 0 \le m <k_p \\ k_p & \text{if } m \ge k_p \end{cases}
\]
and one can show
\[
\sum_{m=0}^{\infty} \frac{ \# \{ [J]: J\le \of{}, \, N(J)=p^m\} }{p^{ms}} = (1-p^{-k_p s})(1-p^{-s})^{-2}.
\]
Therefore $\slzeta{L,p}{s}= (1-p^{-k_p s}) (1-p^{-s})^{-1} (1-p^{1-s})^{-1}$.
\end{proof}

It turns out that
\[
a_{p^i}^+(L) = \begin{cases} \frac{p^{i+1}-1}{p-1} & \text{if } p \text{ inert or ramified}, \\
\frac{p^{i+1} - p^{\mathrm{max}(i-k_p+1,0)}}{p-1} & \text{if } p \text{ split}. \end{cases}
\]
Hence one may recover all sets of split primes of $F$ from $\slzeta{L}{s}$. Since the set of split primes of a quadratic field $F$ determines $F$, the following corollary holds.

\begin{cor}\label{cor:slzeta.F}
The series $\slzeta{L}{s}$ is uniquely determined by the quadratic field $F$.
\end{cor}

\end{document}